\documentclass[11pt]{article}
\usepackage{amsmath,amssymb,amsthm,nicefrac,mathrsfs,bm,setspace,amsfonts}
\usepackage[titletoc]{appendix}
\usepackage[colorlinks]{hyperref}

\usepackage[a4paper, margin=3cm, footskip=1cm]{geometry}

\newtheorem{stat}{Statement}[section]
\newtheorem{theorem}{Theorem}[section]
\newtheorem{lemma}[theorem]{Lemma}
\newtheorem{proposition}[theorem]{Proposition}
\newtheorem{corollary}[theorem]{Corollary}
\theoremstyle{definition}
\newtheorem{assumption}[stat]{Assumption}
\newtheorem{remark}[stat]{Remark}

\numberwithin{equation}{section}

\newcommand{\be}{\begin{equation}}
\newcommand{\ee}{\end{equation}}

\def\E{\bE}
\def\P{\bP} 

\def\bE{\mathbb{E}}

\def\bP{\mathbb{P}}

\newcommand{\N}{\mathbf{N}}
\newcommand{\Z}{\mathbf{Z}}
\newcommand{\B}{\mathbf{B}}

\newcommand{\R}{\mathbf{R}}

\renewcommand{\d}{{\rm d}}

\renewcommand{\ge}{\geqslant}
\renewcommand{\le}{\leqslant}

\newcommand{\bes}{\begin{equation*}}
\newcommand{\ees}{\end{equation*}}

\newcommand{\lip}{\text{\rm Lip}_\sigma }

\renewcommand{\P}{\mathrm{P}}
\renewcommand{\E}{\mathrm{E}}

\def\ind{\mathbf{1}}

\def\mE{\mathbf{E}}

\def\mP{\mathbf{P}}
\def\m1{\mathbf{1}}

\allowdisplaybreaks[1]

\begin{document}
\author{Mathew Joseph}

\title{An invariance principle for the stochastic heat equation
}
\date{}

\maketitle

\begin{abstract}
We approximate the white-noise driven stochastic heat equation by replacing the 
fractional Laplacian by the generator of a discrete time random walk on the one dimensional 
lattice, and approximating white noise by a collection of i.i.d. mean zero random variables.  As a
consequence, we give an alternative proof of the weak convergence of the scaled partition
function of directed polymers in the intermediate disorder regime, to the stochastic heat equation; an advantage of the proof is that it gives the
convergence of all moments.

	
%
\end{abstract}

\section{Introduction}
Consider the following discrete space-time stochastic heat equation
\be \label{eq:dis:she}
u_{i+1}(k)-u_i(k)= \sum_{l\in \Z} \P(k,l) \cdot \big[ u_i(l) -u_i(k)\big]+ \sigma\left(u_i(k)\right)\cdot \xi_i(k)
\ee
where $\P$ is the transition kernel of a random walk $X_n$ on $\Z$ and $\sigma:\R\to \R$ is a Lipschitz continuous
function with Lipschitz coefficient $\lip$. The collection $\xi:=\{\xi_i(k), \, i \in \Z, \, k \in \Z_+\}$ consists of i.i.d. random variables
with 
\begin{equation}\begin{split} \label{cond:xi}
& \bE \,\xi_i(k)=0, \quad \bE \,\xi_i^2(k)=1,\\
&\bE\, \big|\xi_i(k)\big|^{2+\kappa} < \infty \; \text{for some } \kappa>0.
\end{split}\end{equation}
Since $\sum_l \P(k,l)=1$ we can remove $u_i(k)$ from both sides of \eqref{eq:dis:she}. A solution to \eqref{eq:dis:she} satisfies
\be \label{eq:dis:sol}
u_{i+1}(k) = \sum_l \P_{i+1}(k,l) \, u_0(l) + \sum_{j=0}^i \sum_l \P_{i-j}(k,l) \cdot \sigma\big(u_j(l)\big)\cdot\xi_j(l). 
\ee
Above $\P_{i}(k,l)$ gives the $i$-step transition probability of jumping from $k$ to $l$. Note by homogeneity 
$\P_{i-j}(k,l)= \P_{i-j}(0, l-k)$ which we define to be $\P_{i-j}(l-k)$ by an abuse of notation.

We shall assume that the random walk $X_n$ has an asymptotic speed $\mu$ and when centered 
it is in the domain of attraction of a symmetric Stable($\alpha$) process with generator $-\nu(-\Delta)^{\alpha/2}$. Our assumptions will be stated in terms of the characteristic function of $X_1$ in the next section. 
A consequence of our main result Theorem \ref{thm:main}, is that after an appropriate scaling of space, time and noise 
in \eqref{eq:dis:she}, we get the continuous space-time stochastic heat equation. To be precise
consider
\be \label{eq:dis:sc}
u^{(n)}_{i+1}(k)= \sum_{l\in \Z} \P(k,l) \, u_i^{(n)}(l) + \sigma\big(u_i^{(n)}(k)\big)\cdot \frac{\xi_i(k)}{n^{(\alpha-1)/2\alpha}}.
\ee
We show that for $t\in \R_+, \, x \in \R$ 
\[
u^{(n)}_{[nt]}\big([xn^{1/\alpha}] -[\mu nt] \big) \Rightarrow v_t(x),
\]
where $\Rightarrow$ denotes convergence in distribution, and $v_t(x)$ is the mild solution to the stochastic heat equation
\be \label{eq:she}
\partial_t v = -\nu(-\Delta)^{\alpha/2}\,v_t(x) + \sigma(v_t(x)) \cdot \dot{W}(t,x).
\ee
Note that the existence and uniqueness of mild solutions to \eqref{eq:she} (see \cite{foon-khos-09}), as well 
as the finiteness of the speed of the random walk $X_n$ in the domain of attraction of 
a Stable($\alpha$) process requires 
\[ 1<\alpha\le 2,\]
and we shall assume this without further mention for the rest of the paper.

In the case that $\sigma(x)\equiv 1$, the limiting random field \eqref{eq:she} is Gaussian.
One can then use the Lindeberg-Feller theorem to prove the weak convergence of $u^{(n)}$ to 
$v$, and this was the approach
used by \cite{sepp-zhai} in their analysis of the Harness process. Another special case is 
when $\sigma(x)\propto x$. In this case \eqref{eq:she} is called
the parabolic Anderson model or the multiplicative stochastic heat equation. Here
one can explicitly write the solution to \eqref{eq:she} as an infinite series involving multiple 
stochastic integrals (the {\it Wiener chaos expansion}). To illustrate, suppose that $v_0(\cdot)\equiv 1$,
then one can check that
\[v_t(x)= 1+ \sum_{k=1}^\infty \int_{\Delta_k} \int_{\R^k} \prod_{i=1}^k p_{t_i-t_{i-1}}(y_i-y_{i-1}) W(dt_{i-1}\, dy_{i-1}),\]
where $\Delta_k=\{(t_0, t_1,\cdots, t_{k-1}): 0<t_0<t_1<\cdots< t_{k-1}< t =t_k\}$ and $y_k=x$. Similarly
 $u^{(n)}$ also has an infinite series expansion involving sums of products of the random variables $\xi$. 
The approach in \cite{albe-khan-quas} and \cite{cara-sun-zygo} shows the weak convergence of $u^{(n)}$ to 
$v$ by showing the convergence of the individual terms in the infinite series expansions. Both these methods are specific and do not work in general, even for a slight perturbation of $\sigma$.
The approach that we take to address the general case has been inspired by \cite{kana}, where discrete
approximations to stochastic differential equations were considered. We point out that as opposed to weak approximations,
there are quite a few papers concerned with strong approximations of the stochastic heat equation (see \cite{foon-jose-li}, \cite{funa}, \cite{gyon-98}, \cite{gyon-99}, \cite{jose-khos-muel}).

Our results allow us to give an alternative proof of the weak convergence of the 
partition function of {\it one} dimensional directed random polymers in the intermediate disorder regime, obtained in \cite{albe-khan-quas}, \cite{cara-sun-zygo} by different arguments. 
For each $n$ consider  the following measure on paths $\mathbf x= (0=x_0, x_1, x_2, \cdots, x_n)$ of length $n$, where each $x_i \in \Z$:
\[\P^{\xi, \beta}_n (\mathbf x) = \frac{\exp\big[\beta \sum_{i=0}^n \xi_i( x_i)\big]}{Z_n^\xi(\beta)} \, \P_n(\mathbf x) .\] 
The measure $\P_n$ is a probability measure on paths of length $n$, of a random walk following the transition probability $\P$ and starting at $0$. The partition function for the inverse temperature $\beta$ 
\be \label{partition}Z_n(\beta)= Z^{\xi}_n(\beta):= E_0 \exp\big[\beta \sum_{i=0}^n \xi_i(X_i) \big]\ee
is the normalizing constant so that $\P_n^{\xi, \beta}(x)$ is a probability measure on paths of length $n$; here $E_0$ is the expectation
over paths of the random walk starting at $0$. One is 
then interested in the behaviour of paths sampled from this new measure $\P_n^{\xi, \beta}$. As the temperature
becomes large, alternatively $\beta$ becomes small, one expects that the paths behave like that of a
random walk with transition matrix $\P$. On the other hand as the temperature becomes small the effect of the environment
$\xi$ becomes important, since substantially more weight is on paths with larger energy $H_n(\xi) = \sum_{i=0}^n \xi_i(x_i)$.
A big motivation for work in this area is to understand this competition between entropy (from the measure $\P_n$) and energy
$H_n(\xi)$. The competition
between entropy and energy is characterized by the limit of the martingale
\[M_n=\frac{Z_n(\beta)}{(\bE[e^{\beta \xi}])^{n+1}}.\]
If the limit is positive a.e. we say that the polymer measure is in {\it weak} disorder while it is in {\it strong} disorder
if the limit is $0$ a.e.. The polymer paths are diffusive in the weak disorder regime (see \cite{come-yosh}) while we see vastly different behavior in the strong disorder regime.  One (as well as two) dimensional polymers are in the strong disorder regime for any $\beta>0$, while in higher spatial dimensions, the polymer is in the strong disorder regime only if $\beta$ is large enough.  
 It is believed that for one dimensional walks which are in the domain of attraction of Brownian motion the 
typical fluctuation of the polymer paths for any $\beta>0$ would be of order $n^{2/3}$ (as opposed to $n^{1/2}$ when $\beta=0$), although this question is still open in general.
One also expects that $\log Z_n$ has fluctuations of order
$n^{1/3}$ (as opposed to $0$ when $\beta=0$), but this question is also open in general.  We refer to \cite{albe-khan-quas}, \cite{come}, \cite{den} for a more detailed
exposition on this subject. 

In this paper we shall consider (as in \cite{albe-khan-quas} and \cite{cara-sun-zygo}) one dimensional polymers in the intermediate disorder regime, that is when the inverse temperature goes to $0$ at
a specified rate with $n$. With the choice of the inverse temperature $\beta(n)= \beta/n^{(\alpha-1)/2\alpha}$, we observe interesting behavior which we shall detail in the following section.


\subsection{Main results}
Let us denote by $\mu:=\sum_{k\in \Z} k\cdot \P(k)$ the mean of $X_1$. Let $\phi$ denote the characteristic function associated to the random walk
\[
\phi(z) := \sum_{k \in \Z} e^{\mathrm{i}zk} \cdot\P(k),\quad z\in [-\pi,\pi].
\]
The centered characteristic function is 
\[
\tilde{\phi}(z):=e^{\mathrm{-i} \mu z}  \phi(z).
\]
We shall make the following assumption on the random walk.
\begin{assumption} \label{cond1} Assume that $\{z\in [-\pi,\pi]: \vert\phi(z)\vert=1\}=\{0\}$ and 
that there exists a constant $0<a<1$ such that 
\be\label{eq:char}
\tilde{\phi}(z)= 1-\nu|z|^\alpha + \mathcal{D}(z),
\ee
where $\mathcal{D}(z)= O(|z|^{\alpha+a})$ as $|z| \to 0$.
\end{assumption}
\begin{remark} The first part of the assumption means that the random walk is {\it strongly aperiodic} and the assumption
on $\tilde \phi$ necessarily implies that the centered random walk is in the domain 
of attraction of the strictly Stable($\alpha$) law.
\end{remark}
\begin{remark} The strong aperiodicity assumption is used in the local limit theorem \ref{llt} below and it should be possible to modify the theorem to remove this assumption. In particular the results of this paper are valid for the simple random walk although
it is not strongly aperiodic.
\end{remark}
Let us now state the main result of this paper. Let
\be \label{u:bar}
\bar u_t^{(n)}(x) := u^{(n)}_{[nt]}\big([xn^{1/\alpha}] -[\mu nt]\big).
\ee
We choose our initial profile 
\be \label{eq:initial} u_0^{(n)}(k) = v_0\Big(\frac{k}{n^{1/\alpha}}\Big)\ee
for a continuous $v_0$. We allow $v_0$ to be random but it should be {\it independent} of the 
white noise $\dot W$ and the random field $\xi$. Our main result is the following 
\begin{theorem} \label{thm:main}
Let the conditions in Assumption \ref{cond1} hold, and fix $m\in \N$ so that $2m<2+\kappa$. Let $v_0$ be a continuous (random) function so that $\sup_x \mE |v_0(x)|^{2m}<\infty$, that is independent of $\xi$ and $\dot W$. Let $u^{(n)}$ be the solution to 
\eqref{eq:dis:sc} with initial profile $u_0^{(n)}$. Then for each $t> 0, \, x\in \R$ we have $\bar u_t^{(n)}(x) \Rightarrow v_t(x)$, where $v$ is the solution to \eqref{eq:she} with initial profile $v_0$. Furthermore we have 
\be \label{eq:main}
\mE\big|\bar u^{(n)}_t(x) \big|^{2m} \to \mE\big| v_t(x)\big|^{2m} \quad \text{as } n \to \infty.
\ee
\end{theorem}
\begin{remark} Above and for the rest of the paper we use $\mE$ to denote the expectation over all random quantities involved. For ease of notation, we will often drop the superscript $n$ when referring from $u^{(n)}$ and $\bar u^{(n)}$. 
\end{remark}
\begin{remark} The proof also shows the convergence of all moments of order less than $2m$. Our proof can also be used to give an upper bound on the rate of convergence in \eqref{eq:main}. This will of course
depend on the 
initial profile $v_0$. 
\end{remark}
\begin{remark} In the case that $\bE \xi_i^2(k)=\rho>0$ one has a factor of $\sqrt\rho$ with the last term 
in \eqref{eq:she}.
\end{remark}
We next focus on the case that $\sigma$ is uniformly bounded, that is $\sup_x |\sigma(x) | < \infty$. In this
situation we can consider more general initial profiles, not necessarily those that are uniformly bounded
in $L^{2m}$. We state a general result for the special case $\alpha=2$.  Let $\eta(k), \, k \in \Z$ be i.i.d.
random variables such that
\be \label{eta}
\mE\, \eta(k)=0,\quad \mE\, \eta^2(k)=\lambda,\quad \mE\, \vert \eta(k) \vert^{2+\kappa'}< \infty \text{ for some } \kappa'>0, 
\ee 
We further assume that the random variables $\eta= \{ \eta(k),\, k \in \Z\}$ are independent of the noise $\xi$.
Let 
\[ u_0^{(n), \eta}(l) = \begin{cases} 0 & \text{if } l=0,\\ n^{-1/4} \sum_{k=1}^l \eta(k) & \text{if } l>0, \\ n^{-1/4} \sum_{k=l}^{0} \eta(k) &\text{if } l <0.\end{cases}\]
We have the following 
\begin{theorem}[$\alpha=2$] \label{thm:char} Assume that $\sup_x |\sigma(x)|<\infty$, and the random variables $\eta$
satisfy \eqref{eta} and are independent of $\xi$ and $\dot W$. Let the conditions in Assumption \ref{cond1} hold, and fix $m\in \N$ such that $2m< 2+\min(\kappa, \kappa')$. Let $u^{(n)}$ be the solution to 
\eqref{eq:dis:sc} with initial profile $u_0^{(n), \eta}$. Then for each $t> 0, \, x\in \R$ we have $\bar u_t^{(n)}(x) \Rightarrow v_t(x)$, where $v$ is the solution to \eqref{eq:she} with initial profile $B$, a two-sided Brownian motion with variance $\lambda$. Furthermore we have 
\bes
\mE\big|\bar u^{(n)}_t(x) \big|^{2m} \to \mE\big| v_t(x)\big|^{2m} \quad \text{as } n \to \infty.
\ees
\end{theorem}
\begin{remark} 
The above theorem explains the height fluctuations of the harness process obtained in \cite{sepp-zhai}. This is the system $h_{i+1}(k)=\sum_l \P(k, l) \cdot h_i(l) +\xi_{i}(k),\, i \ge 0, \, k\in \Z.$
Let us consider the simplest case where the initial height profile satisfies $h_0(0)=0$, and the increments $h_0(k+1)-h_0(k), k \in \Z$ are i.i.d. random variables with finite $2+\kappa'$ moments and mean $\rho_0$ (say). Then the transformation
$u_i(k)= n^{-1/4} \cdot [h_i(k)-\rho_0\mu i -\rho_0 k]$ satisfies \eqref{eq:dis:sc} and the initial profile $u_0$ is of the form above. We then have that $n^{-1/4}\cdot \big [h_{[nt]}([x\sqrt n]-[\mu nt]) -\rho_0[x\sqrt n]\big] \Rightarrow v_t(x)$ where $\partial_t v =\nu \partial_x^2 v + \dot{W}$, which implies that the harness process is in the {\it Edwards-Wilkinson} class; see \cite{corw} for an explanation of this class and the related KPZ class.
\end{remark}

We next consider the scaled partition function of the directed random polymer in the intermediate disorder regime:
\be\label{part:sc}  M_n^{(\tilde\xi)} = \frac{E_0 \exp(\beta  \sum_{i=1}^n \tilde\xi_i(X_i))}{[\bE e^{\beta \tilde\xi}]^{n+1}}\ee
where $\tilde\xi$ is the scaled disorder
\[ \tilde\xi_i(k) = \frac{\xi_i(k)}{n^{(\alpha-1)/2\alpha}} .\]
We shall also be interested in the scaled point to point partition function given by 
\be\label{part:sc:point}  M_n^{(\tilde\xi,\,x)} = \frac{E_0\big[ \mathbf 1\{X_n=x\} \cdot \exp(\beta  \sum_{i=1}^n \tilde\xi_i(X_i))\big]}{[\bE e^{\beta \tilde\xi}]^{n+1}}\ee
Note that $  M_n^{(\tilde\xi,\, x)}/M_n^{(\tilde\xi)}=\P_n^{\tilde\xi, \beta}(x)$ is the probability of the intermediate disorder polymer to be at $x$ at time $n$. We rederive the following theorem from \cite{albe-khan-quas} and \cite{cara-sun-zygo}. An advantage of our method is that we are able to show the convergence of all moments; see the first two statements of the theorem.

\begin{theorem}\label{thm:part} Assume that the environment variables $\xi_i(k)$ are i.i.d. and have exponential moments: $\bE e^{C|\xi_i(k)|}< \infty$ for all small $C>0$. Suppose that $\mu=0$. Then the following statements hold.
\begin{enumerate}
\item $ M_n^{(\tilde\xi)}$ converges in distribution to $v_1(0)$, where $\partial_t v= -\nu (-\Delta)^{\alpha/2}v +\beta v \dot{W}$ with initial
profile $v_0\equiv 1$. Furthermore the moments of $M_n^{(\tilde\xi)}$ converge to those of $v_1(0)$.
\item $ n^{1/\alpha} M_n^{(\tilde\xi,\,[xn^{1/\alpha}])}$ converges in distribution to $g_1^{(x)}(0)$, where $\partial_t g^{(x)}= -\nu (-\Delta)^{\alpha/2}g^{(x)} +\beta g^{(x)} \dot{W}$ with initial
profile $g_0^{(x)}=\delta_x$. Furthermore the moments of $\ M_n^{(\tilde\xi,\,[xn^{1/\alpha}])}$ converge to those of $g_1^{(x)}(0)$.
\item $n^{1/\alpha}\P_n^{\tilde \xi, \beta} ([xn^{1/\alpha}])$ converges in distribution to $g_1^{(x)}(0)/v_1(0)$, where $g^{(x)}$ and $v$ are as above and driven by the same white noise $\dot W$. 
\end{enumerate}
\end{theorem}
\begin{remark}
One can show that $\int [g_1^{(x)}(0)/v_1(0)] \,\d x=1$.
\end{remark}
The first statement in the theorem implies that $\log Z_n$ has fluctuations of order $1$ in the intermediate disorder regime. The last statement says that the polymer paths have fluctuations of order $n^{1/\alpha}$ and satisfy a {\it random} local limit theorem. Recall that when $\alpha=2$ and the disorder is not scaled one expects $n^{1/3}$ fluctuations for $\log Z_n$ and $n^{2/3}$ fluctuations for the polymer
paths.

We end this section with a brief description of the notation and outline of the paper. We will use indices $i,j \in \Z_+$ for discrete time, $k,l \in \Z$ for discrete space. The letters $q,r,s,t$ will usually be used for continuous time $\R_+$,while the variables $x,y,z,w$ will usually be used  for continuous space $\R$. The notation $\bP, \bE$ is used to denote probability and expectation for the random variables $\xi$, while $P, E$ will be used for the probability and expectation over paths of random walks with transition kernel $\P$. As already indicated $\mP, \mE$ will be used when we integrate out all the random variables involved. We will denote by $\|\cdot \|_m$ the norm in $L^m(\mP),\, 1\le m<\infty$. Constants $C$ could vary from line to line. We will use
the notation $c, c_1, c_2, \cdots $ to represent constants which remain constant within a proof but might change
across different proofs. 

In Section \ref{sec:exis} we prove existence and uniqueness for solutions to \eqref{eq:dis:sc}, and obtain moment bounds required in the proof of Theorem \ref{thm:main} Section \ref{sec:main} deals with the proof of our main Theorem \ref{thm:main}. Section \ref{sec:char} deals with the proof of Theorem \ref{thm:char}. In Section \ref{sec:tight} we provide conditions for tightness to hold in
Theorems \ref{thm:main} and \ref{thm:char}. The proof of the first statement in Theorem \ref{thm:part} is in Section \ref{sec:part:1}. Section \ref{sec:ext} is devoted to several extensions including addition of a drift in \eqref{eq:dis:sc}, Dirac initial condition in \eqref{eq:dis:sc} and the proofs of the last two statements in Theorem \ref{thm:part}. The appendices contain a local limit theorem crucial for our arguments, as well as bounds needed for the proof of Proposition \ref{prop:u-U}.

\section{Existence and Uniqueness} \label{sec:exis}
In this section we show the existence and uniqueness of solutions to \eqref{eq:dis:sc}. The ideas are quite 
standard, see for example \cite{foon-khos-12}, but we present them here to show that we can obtain 
moment bounds which are uniform in $n$. We shall need these in the proof of Theorem \ref{thm:main}. Recall that 
a solution to \eqref{eq:dis:sc} satisfies
\be \label{eq:dis:sc:sol}
u_{i+1}(k) = \sum_l \P_{i+1}(k,l) \, u_0(l) + \sum_{j=0}^i \sum_l \P_{i-j}(k,l) \cdot \sigma\big(u_j(l)\big)\cdot\frac{\xi_j(l)}{n^{(\alpha-1)/2\alpha}}. 
\ee
We call the second term the {\it noise term} because it depends on the noise $\xi$. We call the first term the {\it non-noise term} although the initial profile $u_0$ is allowed to be random.
\begin{theorem} \label{thm:u} For initial profiles $u_0 = u_0^{(n)}$ to \eqref{eq:dis:sc} with $\sup_{k, n} \mE |u_0(k)|^2<\infty$, the following statements hold.
\begin{enumerate}
\item There exists a unique solution to equation
\eqref{eq:dis:sc} such 
that for any $T>0$ 
\be\label{eq:mom:unif}
\sup_{i\le [n T], \, k \in \Z} \mE | u_i(k) |^2 <\infty.
\ee
Furthermore the above bound holds uniformly in $n\ge 1$.
\item Let $m\in \N$ and suppose further that $\xi$ has $2m$ moments. Assume also that $\sup_{k, n} \mE\big|u_0(k)\big|^{2m}<\infty$. Then for any $T>0$ the following holds uniformly in $n \ge 1$
\be \label{eq:thm:um}
\sup_{i\le [n T], \, k \in \Z} \mE\big| u_i(k) \big|^{2m} <\infty.
\ee
\end{enumerate}
\end{theorem} 

\begin{proof}
We use Picard's iteration scheme to show the existence 
of a solution which satisfies the required bound. Therefore let $w^{(0)}_i(k)= u_0(k)$ and let
\be \label{eq:picard}
w^{(p+1)}_{i+1}(k) = \sum_l \P_{i+1}(k,l) u_0(l) + \sum_{j=0}^i \sum_l \P_{i-j}(k,l) \cdot\sigma\big(w^{(p)}_j(l)\big)\cdot\frac{ \xi_j(l)}{n^{(\alpha-1)/2\alpha}} 
\ee 
In the case $\lip=0$ the convergence of the Picard iterates is immediate. So let us assume that $\lip>0$. From the above one has 
\[
\mE \left|w^{(p+1)}_{i+1}(k)-w^{(p)}_{i+1}(k)\right|^2 \le  \lip^2\sum_{j=0}^{i} \sum_{l \in \Z}\P_{i-j}^2(l-k)\cdot  \frac{\mE \big|w^{(p)}_{j}(l)-w^{(p-1)}_j(l)\big|^2}{n^{(\alpha-1)/\alpha}}.
\]
For a parameter $\delta>0$ define
\be \label{w}
\mathcal{W}^2(p) =  \sup_{k \in \Z,\, i\le [nT]} e^{-\delta i/n}\cdot \mE\left |w^{(p)}_{i}(k)-w^{(p-1)}_i(k)\right|^2.
\ee
The above inequality gives
\be\label{eq:w}
\mathcal{W}^2(p+1) \le \lip^2\cdot \mathcal{W}^2(p)\sum_{j=0}^{[nT]}    \frac{e^{-\delta j/n}\cdot P \big(X_{j}=\tilde{X}_{j}\big)}{n^{(\alpha-1)/\alpha}},
\ee
where $X$ and $\tilde X$ are independent walks with the same transition probabilities $\P$. We now choose $\delta$ to be large enough
\[
\sum_{i=0}^{[nT]} \frac{e^{-\delta i/n}\cdot P \big(X_{i}=\tilde{X}_{i}\big)}{n^{(\alpha-1)/\alpha}} <\frac{1}{2\lip^2}.
\]
The reason we can choose such a $\delta$ is that 
\be \label{eq:mom2:u}
\sum_{i=0}^{[nT]} \frac{P \big(X_{i}=\tilde{X}_{i}\big)}{n^{(\alpha-1)/\alpha}} <\infty
\ee
uniformly in $n\ge 1$ by Corollary \ref{cor:green} (see Appendix \ref{sec:app:A}) applied to $X-\tilde X$; a similar statement to Theorem \ref{llt} holds for $X-\tilde X$.
From here it follows that for this choice of $\delta$
\[
\sum_{ p\ge 1} \mathcal{W}(p) <\infty,
\]
uniformly over $n$. This implies the convergence of $w^{(p)}$ to a random field $u^{*}$ uniformly in the interval $k\in \Z,\, i\le [nT]$. One can then argue that $u^{*}$ satisfies \eqref{eq:dis:sc} in this region by considering the limit of both sides of \eqref{eq:picard}. 
To prove uniqueness assume that there are two solutions $u$ and $\tilde u$ to \eqref{eq:dis:sc} satisfying \eqref{eq:mom:unif}.
Then 
\begin{equation*}
\mE \left|u_{i+1}(k)-\tilde u_{i+1}(k)\right|^2 \le  \lip^2\sum_{j=0}^{i} \sum_{l \in \Z}\P_{i-j}^2(l-k)\cdot  \frac{\mE \big|u_{j}(l)-\tilde u_j(l)\big|^2}{n^{(\alpha-1)/\alpha}}.
\end{equation*}
As before multiply both sides by $e^{-\delta i/n}$, and take supremum
over $k\in \Z$ and $i\le [nT]$ to get a relation similar to \eqref{eq:w}. If we choose $\delta$ to be large enough we get $\mE\vert u_i(k)-\tilde u_i(k)\vert^2=0$ for $k\in \Z,\, i\le [nT]$, thus proving uniqueness.

Let us now prove the second part of the theorem. An application of Burkholder's inequality for discrete time martingales
gives
\be\begin{split} \label{eq:u:m}
&\big\| u_{i+1}(k) \big\|_{2m}^2 \\
&\le c_1+c_2 \bigg\| \sum_{j=0}^i \sum_{l,l'\in \Z} \P_{i-j}(l-k) \P_{i-j}(l'-k) \cdot \sigma\big(u_j(l)\big) \sigma\big(u_j(l')\big)\cdot \frac{\xi_j(l)\xi_j(l')}{n^{(\alpha-1)/\alpha}}\bigg\|_{m} \\
& \le  c_1+c_2 \sum_{j=0}^i \bigg\|\left[\sum_{l\in \Z} \P_{i-j}(l-k) \cdot \sigma\big(u_j(l)\big)\cdot \frac{\xi_j(l)}{n^{(\alpha-1)/2\alpha}}\right]^2\bigg\|_{m},
\end{split}\ee
the last step follows from Minkowski's inequality. Now 
\be \begin{split}\label{eq:mom:4m}
&\mE \left | \left[\sum_{l\in \Z} \P_{i-j}(l-k) \cdot \sigma\big(u_j(l)\big)\cdot \frac{\xi_j(l)}{n^{(\alpha-1)/2\alpha}}\right]^2\right|^{m} \\
&= \mE \bigg[ \sum_{l_1,l_2,\cdots l_{2m}} \prod_{a=1}^{2m} \left\{ \P_{i-j}(l_a-k) \cdot \sigma\big(u_j(l_a)\big) \cdot \frac{\xi_j(l_a)}{n^{(\alpha-1)/2\alpha}}\right\} \bigg].
\end{split}\ee
Clearly $\xi_i(k)$ are independent of $u_j(l), \, l \in \Z,\, j\le i$. Consider now the expectation of each product term in the above expression. Our assumption that  $\bE \xi_j(l)=0$ implies that the only
product terms which have a nonzero contribution contain only powers of  $\xi_j(l)$ greater than $1$,  for any $l$. Thus the above 
can be bound by a constant multiple of
\[
\left(\frac{1}{n^{(\alpha-1)/\alpha}}\right)^{m}\cdot\sum_{b=1}^{m}\; \sum_{\substack{(\beta_1,\cdots, \beta_b) \in \mathbf{N}^b\\ \text{each } \beta_q \ge 2\\ \beta_1+  \cdots \beta_b=2m}} \;\sum_{l_1,l_2, \cdots, l_{b}} \prod_{a=1}^b\left\{ \P_{i-j}^{\beta_a}(l_a-k) \right\} \cdot \left\{ \mE\left[\prod_{a=1}^b\left |\sigma^{\beta_a}\left(u_j(l_a)\right)\right|\right] \right\}.
\]
Note also by Holder's inequality, for any $l_1,l_2,\cdots, l_{b} \in \Z$
\[
\begin{split}
 \mE\left[\prod_{a=1}^b\left |\sigma^{\beta_a}\left(u_j(l_a)\right)\right|\right] & \le \prod_{a=1}^{b}\left\|\sigma\big(u_j(l_a)\big)\right\|^{\beta_a}_{2m}.
\end{split}
\]
We also need the following 
\[
\prod_{a=1}^b\left(\sum_{l_a}\P_{i-j}^{\beta_a}(l_a-k)\left\|\sigma\big(u_j(l_a)\big)\right\|^{\beta_a}_{2m}\right) \le \prod_{a=1}^b\left(\sum_{l_a}\P_{i-j}^{2}(l_a-k)\left\|\sigma\big(u_j(l_a)\big)\right\|^{2}_{2m}\right)^{\beta_a/2}, 
\]
which follows from the relation $\sum x_i^r \le (\sum x_i)^r$ valid for any $r\ge 1$ and nonnegative sequences $x_i$. These observations imply that \eqref{eq:mom:4m} can be bound by a constant multiple of 
\[
\left(\frac{1}{n^{(\alpha-1)/\alpha}}\right)^{m}\cdot \left( \sum_l \P_{i-j}^2(l-k)\cdot \left\|\sigma\big(u_j(l)\big)\right\|_{2m}^2\right)^{m}.
\]
If we plug this in \eqref{eq:u:m} one gets
\begin{equation*} \begin{split}
\big\| u_{i+1}(k) \big\|_{2m}^2 &\le c_1+c_3\sum_{j=0}^i \sum_l \P_{i-j}^2(l-k) \frac{ \left\|\sigma\big(u_j(l)\big)\right\|_{2m}^2}{n^{(\alpha-1)/\alpha}}\\
&\le c_1+ c_4 \sum_{j=0}^i \sum_l \P_{i-j}^2(l-k)\cdot \frac{1+ \left\|u_j(l)\right\|_{2m}^2}{n^{(\alpha-1)/\alpha}}.
\end{split}\end{equation*}
Now, for a fixed $\delta>0$ let 
\[
\mathcal{X}(i) :=e^{-\delta i/n} \sup_{k \in \Z} \|u_i(k)\|_{2m}^2.
\]
Our calculations above give us
\be\label{eq:recursive}
\mathcal X(i+1) \le c_1+c_4 \sup_{q\le i} \left[1+\mathcal{X}(q)\right] \cdot \sum_{j=0}^{[nT]}\frac{e^{-\delta j/n} P(X_{j}=\tilde{X}_{j})}{n^{(\alpha-1)/\alpha}}.
\ee
By choosing a $\delta$ large enough so that 
\[
c_4 \sum_{j=0}^{[nT]}\frac{e^{-\delta j/n} P(X_{j}=\tilde{X}_{j})}{n^{(\alpha-1)/\alpha}}<\frac12,
\]
one can obtain \eqref{eq:thm:um} by a recursive application of \eqref{eq:recursive}. This
completes the proof of the theorem.
\end{proof}

\section{Proof of Theorem \ref{thm:main}} \label{sec:main}

Fix $m\in \N$ so that $2m <2+\kappa$. For each $n$ we shall construct a probability space containing copies of the random variables $\xi$, the
white noise $\dot W$, and the initial profile $v_0$ independent of $\xi$ and $\dot W$ so that 
\be \label{eq:couple} \mE\big[ \big\vert  \bar u_t(x) - v_t(x) \big\vert^{2m}\big] \to 0  \quad \text{ as } n \to \infty.  \ee
This would imply the weak convergence stated in the main theorem.

{\it We first prove the result under the assumption that the initial profile $v_0\equiv 0$}. This shall be relaxed
later. The proof will involve several steps which we divide into subsections. To find an upper bound on the rate of convergence in \eqref{eq:couple} we shall need to optimize several quantities over $0<\gamma<\theta< \frac1\alpha$ such that 
\be \label{g:t:cond}
\theta+\gamma< \frac{\min(a,\alpha-1)}{\alpha} \quad \text{and} \quad  \gamma< \Big[(\alpha-1)\wedge \frac1\alpha\Big] \theta.
\ee

\subsection{Adding $\xi$ over blocks}
In the first step we apply a coarse graining
procedure and compare the solution to \eqref{eq:dis:sc} to a random field on a coarser lattice.  Each site in the coarser lattice would correspond to a block of size $[n^\theta]\times[n^\gamma]$
in the original lattice. More precisely site $(j,l)$ (we shall use the first coordinate for time and the second for space) in the coarser lattice will correspond to 
\be \label{eq:jl}
\B_{j}(l):=\big[\,j[n^\theta], (j+1)[n^\theta]\,\big) \times \big[\, l[n^\gamma], (l+1)[n^\gamma]\,\big)
\ee 
in the original lattice. For $i\in \Z_+,\, k \in \Z$ define
\be \label{eq:U}
U_{i[n^\theta]}\big( k [n^\gamma] \big) =  \sum_{j=0}^{i-1} \sum_{l \in \Z} \P_{(i-1-j)[n^\theta]}\big( (l-k)[n^\gamma]\big)\cdot \sigma\big(u_{j[n^\theta]}( l [n^\gamma])\big) \cdot \zeta_j (l),
\ee
where the random variables $\zeta$ are obtained by summing the contribution of the scaled $\xi$ variables
in these blocks:
\be\label{eq:zeta}
\zeta_j(l) = \sum_{(s,y) \in \B_j(l)}
\frac{\xi_s(y)}{n^{(\alpha-1)/2\alpha}}.
\ee
Here again we are suppressing the
dependence of $U=U^{(n)}$ on $n$ in the notation.

The main result of this section is the following theorem where we show that $u$ is close to $U$.
 Choose $ r\in \Z_+$  and $z \in Z$ such that 
\be\begin{split} \label{eq:r:z} &r[n^\theta] \le [nt]<(r+1)[n^\theta], \\
&z[n^\gamma] \le [ xn^{1/\alpha}]-[\mu nt ] <(z+1)[n^\gamma]. 
\end{split}\ee
Recall the definition of $\bar u$ in 
\eqref{u:bar}. 
\begin{theorem} \label{thm:u-U} Fix $T>0$. We have for $\frac{[n^\theta]}{n} \le t\le T, \, x \in \R$
\[ \Big\| \bar u_t(x) - U_{r[n^\theta]}(z[n^\gamma]) \Big\|_{2m}^2 \lesssim \frac{n^{\gamma}}{n^{(\alpha-1)\theta}}+ \frac{n^{\theta+\gamma+o(1)}}{n^{(\alpha-1)/\alpha}}.\]
\end{theorem}
\begin{proof} The proof is obtained by combining Propositions \ref{prop:u:hol} and \ref{prop:u-U} below, and our restrictions on the parameters $\theta$ and $\gamma$ in \eqref{g:t:cond}.
\end{proof}
We first recall a few facts about general one-dimensional recurrent
random walks which we shall apply to $Y=X-\tilde X$, where $X$ and $\tilde X$ are
independent random walks with transition probabilities $\P(k,l)$. Note first that the random walk $Y$ is symmetric. The potential
kernel of the random walk $Y$ is the function
$
\bar a(x) := \lim_{n\to\infty} \bar a_n(x)
$
where
\bes
\bar a_n(x) :=\sum_{j=0}^{n} P(Y_j=0) - \sum_{j=0}^n P(Y_j=x).
\ees
One can show that the function $\bar a(x)$ exists (Propositions 28.5 and 28.8 in \cite{spitzer}) and further that $\bar a_n(x)$ is increasing in $n$ for each $x$. We will need the following lemmas.
\begin{lemma} \label{pker:bd}
The potential kernel $\bar a$ of the $Y$ walk satisfies
\bes
\left|  \bar a(x) \right| \lesssim |x|^{\alpha-1}.
\ees
\end{lemma}
\begin{proof} The characteristic function of $Y_1$ is $\vert \phi(z)\vert^2$. By Fourier inversion
\begin{align*}
\bar a_n(x) &= \frac{1}{2\pi} \int_{-\pi}^{\pi} \frac{1-e^{-\mathrm i x z}}{1-\vert \phi(z)\vert^2}\cdot \Big[1-\vert \phi(z)\vert^{2(n+1)}\Big] \d z \\
&= \frac{1}{2\pi} \int_{-\pi}^{\pi} \frac{1-\cos(x z)}{1-\vert \phi(z)\vert^2}\cdot \Big[1-\vert \phi(z)\vert^{2(n+1)}\Big] \d z 
\end{align*}
By the assumptions on the characteristic function $\phi$, we have the bound
\[ c_1 |z|^\alpha\le1- \vert \phi(z) \vert^2 \le c_2 |z|^\alpha  \]
for $|z| \le \pi$. Therefore uniformly in $n$
\begin{align*}
\bar a_n(x) & \lesssim \int_{-\pi}^{\pi} \frac{\vert 1-\cos(xz)\vert}{|z|^\alpha} \d z 
\end{align*}
which is of order $|x|^{\alpha-1}$.
\end{proof}
\begin{lemma}\label{lem:xunif} The following holds 
\be \label{eq:xunif}
E\left[\left|X_n -\mu n\right|^{\alpha-1}\right] \lesssim n^{(\alpha-1)/\alpha}.
\ee
The above also holds when $X$ is replaced by $Y$ and $\mu$ by $0$.
\end{lemma}
\begin{proof} Let us denote
\[ R_n = \frac{X_n-n \mu }{n^{1/\alpha}}.\]
One can check that for any $\delta>0$
\[ \int_0^\infty \frac{1-\cos(zr)}{z^{1+\delta}} dz= c(\delta) |r|^\delta, \quad r \in \R, \]
for some constant $c(\delta)$. Thus 
\bes
\begin{split}
E\big[|R_n|^\delta\big] &= \frac{1}{c(\delta)} \int_0^\infty \frac{1- E \cos (z R_n)}{z^{1+\delta}}\, dz  \\
&= \frac{1}{c(\delta)} \int_0^\infty \frac{1- \text{Re }E\big(e^{izR_n}\big)}{z^{1+\delta}} \\
&= \frac{1}{c(\delta)} \int_0^\infty \frac{1-\text{Re }\big(\big[\tilde\phi (z/n^{1/\alpha})\big]^n\big)}{z^{1+\delta}}\, dz.
\end{split}
\ees
By our assumption \eqref{eq:char} on the characteristic function one can argue that for $0<\delta<\alpha$ we have
$E[|R_n|^\delta]<\infty$ for all $n$ and further
\[ \lim_{n \to \infty} E\big[|R_n|^\delta\big] =\frac{1}{c(\delta)}\int_0^\infty\frac{1- e^{-\nu |z|^\alpha}}{z^{1+\delta}}\, d z.\]
This of course implies $\sup_n E\big[|R_n|^\delta\big]<\infty$. We can choose $\delta=\alpha-1$ to complete the proof of 
\eqref{eq:xunif}. We have a similar bound for $Y$ with $\mu$ replaced by $0$ because the characteristic
function of $Y_1$ satisfies \eqref{eq:char} with $\nu$ replaced by $2\nu$.
\end{proof} 

The following should be thought of a H\"older continuity estimate for $u$. It is the first step in our proof 
of Theorem \ref{thm:u-U}.
\begin{proposition} \label{prop:u:hol} We have the following bound  valid for $t \ge \frac{[n^\theta]}{n}$
\[ \left\|\bar u_t(x)- u_{r[n^\theta]}(z[n^\gamma])\right\|_{2m}^2 \lesssim \frac{1}{n^{(\frac1\alpha-\theta)(\alpha-1)}} .\]
\end{proposition}
\begin{proof} We write
\begin{align*}
&\bar u_t(x) - u_{r[n^\theta]}(z[n^\gamma])= \sum_{j=r[n^\theta]}^{[nt]-1}\sum_{l\in \Z} \P_{[nt]-j-1}\big(l+[\mu nt] -[xn^{1/\alpha}]\big) \cdot \sigma\big(u_j(l)\big) \cdot \frac{\xi_j(l)}{n^{(\alpha-1)/2\alpha}}  \\
&\quad + \sum_{j=0}^{r[n^\theta]-1} \sum_{l \in \Z} \Big[ \P_{[nt]-j-1}\big(l+[\mu nt] -[xn^{1/\alpha}]\big) - \P_{r[n^\theta]-j-1}\big(l- z[n^\gamma]\big)\Big]\cdot \sigma\big(u_j(l)\big) \cdot \frac{\xi_j(l)}{n^{(\alpha-1)/2\alpha}}. 
\end{align*}
An application of Burkholder's inequality along with the bound in Theorem \ref{thm:u} then gives
\be\begin{split} 
& \left\| \bar u_t(x) - u_{r[n^\theta]}\big(z[n^\gamma]\big) \right\|_{2m}^2 \\
&\lesssim \sum_{j=r[n^\theta]}^{[nt]-1}\sum_{l\in \Z} \frac{\P^2_{[nt]-j-1}\big(l+[\mu nt] -[xn^{1/\alpha}]\big) }{n^{(\alpha-1)/\alpha}} \\
& \quad + \sum_{j=0}^{r[n^\theta]-1} \sum_{l \in \Z} \frac{\left[ \P_{[nt]-j-1}\big(l+[\mu nt] -[xn^{1/\alpha}]\big) - \P_{r[n^\theta]-j-1}\big(l- z[n^\gamma]\big)\right]^2}{n^{(\alpha-1)/\alpha}}  
\end{split}
\ee
For fixed time $j$ and spatial point $k$ the sum $\sum_l \P_j(l) \P_j (l+k) =P(X_j =\tilde X_j +k)$ where $X$ and $\tilde X$ 
are independent random walks. Therefore the above is
\be
\begin{split}\label{eq:u:holder}
&\lesssim \frac{1}{n^{(\alpha-1)/\alpha}}\sum_{j=0}^{[n^\theta]} P\big(X_j=\tilde X_j\big) \\
&\;\;+ \frac{1}{n^{(\alpha-1)/\alpha}}\sum_{j=0}^{r[n^\theta]-1} P\big(X_{[nt]-j-1}=\tilde X_{[nt]-j-1}\big) + \frac{1}{n^{(\alpha-1)/\alpha}}\sum_{j=0}^{r[n^\theta]-1} P\big(X_{r[n^\theta]-j-1}=\tilde X_{r[n^\theta]-j-1}\big) \\
&\;\;- \frac{2}{n^{(\alpha-1)/\alpha}}\sum_{j=0}^{r[n^\theta]-1} P\big(X_{[nt]-j-1}=\tilde X_{r[n^\theta]-j-1}+[xn^{1/\alpha}]-[\mu nt]-z[n^\gamma]\big).
\end{split}\ee
As already observed in \eqref{eq:mom2:u} one can use Corollary \ref{cor:green}
with $Y=X-\tilde X$ to get
\be \begin{split} \label{eq:1}
 \frac{1}{n^{(\alpha-1)/\alpha}}\sum_{j=0}^{[n^\theta]} P\left(X_j=\tilde X_j\right) 
 &\lesssim \frac{1}{n^{(1-\theta)(\alpha-1)/\alpha}}.
\end{split}\ee
We now proceed with \eqref{eq:u:holder}. An application of the Markov property gives
\bes
\begin{split} 
&  \left\|\bar u_t(x) - u_{r[n^\theta]}\big(z[n^\gamma]\big) \right\|_{2m}^2  \\
& \lesssim \frac{1}{n^{(1-\theta)(\alpha-1)/\alpha}} + \frac{1}{n^{(\alpha-1)/\alpha}} \sum_{j=0}^{r[n^\theta]-1}\left[P \big(Y_j = Z_{[nt]-r[n^\theta]}\big) + P(Y_j=0)\right] \\
& \;\;   -  \frac{2}{n^{(\alpha-1)/\alpha}} \sum_{j=0}^{r[n^\theta]-1} P\big(Y_j= W_{[nt]-r[n^\theta]}-[xn^{1/\alpha}]+[\mu nt]+z[n^\gamma]\big),
\end{split}
\ees
where $Z\stackrel{d}{=} Y$  and $W\stackrel{d}{=} X$, and both $Z, W$ are independent of the walks $X$ and $\tilde X$. 
Thus
\be \begin{split} \label{eq:u:holder:2}
& \left\| \bar u_t(x) - u_{r[n^\theta]}\big(z[n^\gamma]\big) \right\|_{2m}^2 \\
& \lesssim \frac{1}{n^{(1-\theta)(\alpha-1)/\alpha}} +  \frac{1}{n^{(\alpha-1)/\alpha}} \left[ E\,\bar a\big(Z_{[nt]-r[n^\theta]}\big)+ E\, \bar a \Big(W_{[nt]-r[n^\theta]}-[xn^{1/\alpha}]+[\mu nt]+z[n^\gamma]\Big)\right] \\
&\lesssim \frac{1}{n^{(1-\theta)(\alpha-1)/\alpha}} +  \frac{1}{n^{(\alpha-1)/\alpha}} \left[ n^{\theta(\alpha-1)} + n^{\gamma(\alpha-1)} \right] \\
& \lesssim \frac{1}{n^{(\frac1\alpha -\theta)(\alpha-1)}}
\end{split}\ee
In the second last inequality we use Lemma \ref{pker:bd} and Lemma \ref{lem:xunif}. Note
also that here we have used that $\gamma\le \theta< \alpha^{-1}$ from \eqref{g:t:cond}.
This completes the proof of the proposition.
\end{proof}

To complete the proof of Theorem \ref{thm:u-U} we estimate the difference between $u_{r[n^\theta]}(z[n^\gamma])$ and $U_{r[n^\theta]}(z[n^\gamma])$ in the following
\begin{proposition} \label{prop:u-U} Fix $T>0$. We have for $1\le r\le [Tn^{1-\theta}]$. 
\[ \left\|u_{r[n^\theta]}(z[n^\gamma]) -U_{r[n^\theta]}(z[n^\gamma])\right\|_{2m}^2 \lesssim \frac{n^{\gamma}}{n^{(\alpha-1)\theta}}+ \frac{n^{\theta+\gamma+o(1)}}{n^{(\alpha-1)/\alpha}}.\]
\end{proposition}
\begin{proof}
Recall the definition of the random field $U$ in \eqref{eq:U}. By writing the $\zeta$ variables as the sum of the $\xi$ variables in the blocks we can write the above difference as
\be \begin{split} \label{limit}
&\sum_{i=0}^{r-1} \sum_{k \in \Z} \sum_{j=i[n^\theta]}^{(i+1)[n^\theta]-1} \;\sum_{l=k[n^\gamma]}^{(k+1)[n^\gamma]-1}\P_{r[n^\theta]-j-1}\big(l- z[n^\gamma]\big)\cdot \sigma\big(u_j(l)\big) \cdot \frac{\xi_j(l)}{n^{(\alpha-1)/2\alpha}} \\
&  - \sum_{i=0}^{r-1} \sum_{k \in \Z} \sum_{j=i[n^\theta]}^{(i+1)[n^\theta]-1} \;\sum_{l=k[n^\gamma]}^{(k+1)[n^\gamma]-1} \P_{(r-1-i)[n^\theta]}\big( (k-z)[n^\gamma]\big)\cdot \sigma\big(u_{i[n^\theta]}( k [n^\gamma])\big)  \cdot \frac{\xi_j(l)}{n^{(\alpha-1)/2\alpha}}.
\end{split}\ee
At this stage it is convenient to split $\sigma\left(u_j(l)\right)$ in the first term as
\bes
\sigma\left(u_j(l)\right)=\left[\sigma\left(u_j(l)\right)-\sigma\left(u_{i[n^\theta]}(k[n^\gamma])\right)\right]+\sigma\left(u_{i[n^\theta]}(k[n^\gamma])\right).
\ees
An application of Burkholder's inequality then gives
\be \label{eq:2}
\begin{split}
&\left\| u_{r[n^\theta]}\big(z[n^\gamma]\big)-U_{r[n^\theta]}\big(z[n^\gamma]\big) \right\|_{2m}^2 \\
& \lesssim \frac{1}{n^{(\alpha-1)/\alpha}} \sum_{i,j} \sum_{k,l}\P^2_{r[n^\theta]-j-1}\big(l- z[n^\gamma]\big)\cdot \left\|u_j(l)-u_{i[n^\theta]}\left(k[n^\gamma]\right)\right\|_{2m}^2  \\
& \;\; + \frac{1}{n^{(\alpha-1)/\alpha}} \sum_{i,j} \sum_{k,l}\Big[\P_{r[n^\theta]-j-1}\big(l- z[n^\gamma]\big)- \P_{(r-1-i)[n^\theta]}\big( (k-z)[n^\gamma]\big)\Big]^2\\
&\lesssim \frac{1}{n^{(\frac1\alpha-\theta)(\alpha-1)}} + \frac{1}{n^{(\alpha-1)/\alpha}} \sum_{i,j} \sum_{k,l} \Big[\P_j^2(l) + \P_{i[n^\theta]}^2\big(k[n^\gamma]\big)-2\P_j(l)\P_{i[n^\theta]}\big(k[n^\gamma]\big)\Big],
\end{split}
\ee
where the limits of the sums are just as earlier; we did this to simplify notation. The 
first term in the last line is a bound on the first term of the previous step, and 
this comes from the H\"older continuity estimate in Proposition \ref{prop:u:hol} and \eqref{eq:mom2:u}.
From this point on we focus on the last term in \eqref{eq:2} which can be written as
\be\begin{split} \label{eq:2*}
&  \frac{1}{n^{\frac{(\alpha-1)}{\alpha}}} \sum_{i,j} \sum_{k,l} \P_j(l) \cdot \Big[ \P_j(l)-\P_{i[n^\theta]}\big(k[n^\gamma]\big)\Big]+\frac{1}{n^{\frac{(\alpha-1)}{\alpha}}} \sum_{i,j} \sum_{k,l} \P_{i[n^\theta]}\big(k[n^\gamma]\big)\cdot \Big[ \P_{i[n^\theta]}\big(k[n^\gamma]\big)-\P_j(l)\Big].
\end{split}\ee
In order to keep the exposition simpler for the reader we bound the above two terms in Lemmas \ref{lem:2_1} and \ref{lem:2_2} in Appendix B. The proof of this Proposition is complete once we take into account our restrictions on $\gamma$ and $\theta$ in \eqref{g:t:cond}.
\end{proof}

\subsection{Coupling with white noise}
 Here we approximate $\zeta_j(l)$ in the block $\B_j(l)$ by replacing it by white noise in the rescaled box 
\be \label{eq:bjl}
\tilde{\B}_j(l)  :=\left[\,\frac{j[n^\theta]}{n}, \frac{(j+1)[n^\theta]}{n}\,\right) \times \left[\,\frac{ l[n^\gamma]}{n^{1/\alpha}},\frac{ (l+1)[n^\gamma]}{n^{1/\alpha}}\,\right),\; j\in \Z_+,\, l \in \Z.
\ee
It is this step which allows us to construct copies of $u$ and $v$ on the same space as in \eqref{eq:couple}. For this we shall need the following coupling theorem.
\begin{theorem} \label{thm:couple}
We can construct random variables $\xi'_i(k),\, i \in \Z_+,\, k \in \Z$ and $\zeta'_j(l),\, j \in \Z_+,\, l \in \Z$ and a white noise $W'$ on the same probability space such that the following hold
\begin{enumerate}
\item $\xi'$ is distributed as $\xi$ and $\zeta'$ is distributed as $\zeta$. 
\item For fixed $j \in \Z_+,\, l \in \Z$, we have that $\big(W'_j(l),\, \zeta'_j(l)\big)$ is independent of the collection $\left(W'_i(k),\, \zeta'_i(k)\right),\, i \ne j, \, k\ne l$ where
\bes
W'_j(l) := W'\left(\tilde{\B}_j(l)\right).
\ees
\item For any real number $m$ such that $0<2m<2+\kappa$ and $j \in Z_+,\, l \in \Z$ we have 
\be \label{eq:z-W}
\mE\left[\left|\zeta_j'(l)- n^{1/\alpha}W_j'(l)\right|^{2m}\right] \lesssim n^{-m\{\alpha-1-\alpha(\theta+\gamma)\}/\alpha}n^{-(\theta+\gamma)\cdot\min(1,\,\kappa)/2}.
\ee
\end{enumerate}
\end{theorem}
\begin{proof} Parts 1 and 2 can be proved similar to Lemma 1 in \cite{kana}. Start with a white noise $\hat W=\left(\hat W(e_1), \hat W(e_2), \cdots\right)$ on $\tilde{B}_0(0)$ constructed on some space $(\Omega, Q)$. Here $e_1,e_2,\cdots$ form a basis of $L^2\left(\tilde B_0(0)\right)$.  Denote by $F$ the 
distribution of $\zeta_j(l)$ and let $\Phi$ denote the standard normal distribution. Define the inverse 
function of $F$ by 
\bes
F^{-1}(x)= \sup_{F(u)\le x} u.
\ees
For a standard normal random variable $Z$ we have $\Phi(Z) \stackrel{d}{=} U(0,1)$ and
therefore 
\bes
\hat\zeta_0(0):=F^{-1}\left(\Phi\left[\sqrt{\frac{n^{1+\frac1\alpha}}{[n^\theta][n^\gamma]}} \;\hat W\left(\tilde{B}_0(0)\right)\right]\right) \stackrel{d}{=} F.
\ees
 This gives a way of coupling $\zeta_0(0)$ and $W_0(0)$. We now construct the random 
variables $\xi$ in the box $B_0(0)$ as follows. Define probability measures $R$ on $\R^{[n^\theta]\cdot [n^\gamma]}\times \R$ and $\bar R$ on $\R\times \R^{\mathbf{N}}$ by
\bes
\begin{split}
R(C\times D) &= P\left(\left(\xi_{i}(k)\right)_{\{(i,k) \in B_0(0)\}} \in C,\, \zeta_0(0) \in D\right) \\
\bar R(D\times E) &= P\left(\hat\zeta_0(0) \in D,\, \hat W \in E\right) 
\end{split}
\ees
Since $R(C\times \cdot)$ and $\bar R(\cdot \times E)$ are absolutely continuous with respect to the probability measure $S(\cdot)=R(\R^{[n^\theta]\cdot [n^\gamma]}\times \cdot)= \bar R(\cdot\times\R^{\mathbf{N}}) $ we have 
\bes\begin{split}
R(C\times D) &= \int_{D} p_{C}(y) S(dy) \\
\bar R(D\times E) &= \int_{D} q_E(y) S(dy)
\end{split}\ees
for measurable functions $p_C$ and $q_E$. We now construct the probability measure $\bar Q$ on $\bar\Omega=\R^{[n^\theta]\cdot [n^\gamma]}\times \R\times \R^{\mathbf{N}}$ by 
\bes
\bar Q( C\times D\times E) := \int_{D} p_C(y) q_{E}(y) S(dy).
\ees
We have constructed random variables $\xi',\, \zeta', \, W'$ corresponding to the box $B_0(0)$. To construct for the whole space we consider $\Z_+\times \Z$ i.i.d copies of $\bar\Omega\times\bar Q$. For any $(j,l)$ the $\zeta$'s, $\xi$'s constructed for the $(j,l)$'th copy would correspond to $B_j(l)$. Similarly the white noise for the $(j,l)$'th copy would correspond to $\tilde B_j(l)$. The white noise on $\R_+\times \R $ is created by ``gluing" together the white noise in the different boxes. To be precise for any $f \in L^2(\R_+\times \R)$ define
\[ W'(f) = \sum_{(j, l) \in \Z_+ \times \Z} W'\big( f \cdot \mathbf 1_{B_j(l)}\big). \] 
One can check that this is a white noise on $\R_+\times \R $.

Let us next prove part 3 of the theorem. Denote by $\tilde F$ the distribution function of the random variable
\bes
Y:=\sum_{(i,k) \in \B_j(l)} \frac{\xi'_i(k)}{\sqrt{[n^\theta][n^\gamma]}} = \frac{n^{(\alpha-1)/2\alpha}}{\sqrt{[n^\theta][n^\gamma]}}\cdot \zeta'_j(l).
\ees
Note that this random variable has mean $0$ and variance $1$ since we have normalized by
the square root of the size of the box $\B_j(l)$. We shall compare $Y$ with 
\bes
Z= \frac{n^{(\alpha+1)/2\alpha}}{\sqrt{[n^\theta][n^\gamma]}} W_j'(l)
\ees
which has the standard normal distribution. A modified version of the Berry Essen theorem (see \cite{boro}, \cite{osip}, \cite{naga}) gives
\bes
\left|\tilde F(x)-\Phi(x)\right| \lesssim \frac{1}{[1+|x|^{2+\kappa}]\cdot n^{(\theta+\gamma)\cdot\min(1,\,\kappa)/2}}, \quad x\in \R.
\ees 
As observed in \cite{ebra} for $2m<2+\kappa$
\bes
\mE|Y-Z|^{2m} \lesssim \int_{\R}  |x|^{2m-1} \left|\tilde F(x)-\Phi(x)\right| \d x  \lesssim \frac{1}{n^{(\theta+\gamma)\min(1,\,\kappa)/2}}.
\ees
The inequality \eqref{eq:z-W} follows easily from this. 
\end{proof}
\begin{remark}From this point on in the proof of Theorem \ref{thm:main} we shall assume that we are working on the probability space constructed above. We shall also
ignore the superscript in the random fields $\xi',\, \zeta'$ and $W'$. 
\end{remark}
In the next approximating process we replace $\zeta$ by the coupled white noise. Again the approximating process will be defined
on the coarser lattice. Let 
\be \begin{split} \label{eq:V}
 V_{i[n^\theta]}\left( k [n^\gamma] \right)  =  \sum_{j=0}^{i-1} \sum_{l \in \Z} \P_{(i-1-j)[n^\theta]}\big( (l-k)[n^\gamma]\big)\cdot \sigma\big(u_{j[n^\theta]}\left(l [n^\gamma]\right)\big) \cdot n^{1/\alpha}W_j (l)
\end{split}\ee
We show now that $V$ is close to $U$. 
\begin{proposition}\label{prop:V-U} Fix any $T>0$. We have uniformly for $1\le r\le [Tn^{1-\theta}],\, k \in \Z$ 
\[\left\|V_{r[n^\theta]}\left( k [n^\gamma] \right)-U_{r[n^\theta]}\left( k [n^\gamma] \right)\right\|_{2m}^2\lesssim \frac{1}{n^{(\theta+\gamma)\cdot \min(1,\,\kappa)/(2m)}} \]
\end{proposition}
\begin{proof} By our assumption \eqref{g:t:cond} on $\theta$ and $\gamma$ we have $\theta+\gamma<(\alpha-1)/\alpha$.
Therefore by Burkholder's inequality and \eqref{eq:z-W}
\bes\begin{split} \label{eq:V-U}
 \left\|V_{r[n^\theta]}\left( k [n^\gamma] \right)-U_{r[n^\theta]}\left( k [n^\gamma] \right)\right\|_{2m}^2 &\lesssim
 \sum_{j=0}^{r-1} \sum_{l \in \Z} \P^2_{(r-1-j)[n^\theta]}\big( (l-k)[n^\gamma]\big)\cdot \frac{n^{\theta+\gamma}}{n^{(\alpha-1)/\alpha}n^{(\theta+\gamma)\cdot\min(1,\,\kappa)/2m}} \\
& \lesssim \frac{1}{n^{(\theta+\gamma)\cdot \min(1,\,\kappa)/2m}}, 
\end{split}\ees
the last line follows because of \eqref{eq:mom2:u} and the fact that second term on the second line of \eqref{eq:2} goes to $0$.
\end{proof}
 
\subsection{Replacing transition probabilities by the Stable($\alpha$) density}
Next we replace the random walk kernel $\P$ by the transition density $p$ for the Stable($\alpha$) process. We can choose integers $a_j$ such that 
\bes
a_j[n^\gamma] \le \mu j[n^\theta]<(a_j+1) [n^\gamma],
\ees
and on the coarser lattice for $i\ge 2,\, k \in \Z$ let 
\be \begin{split} \label{eq:vbar}
  \bar{V}_{i[n^\theta]}\left( k [n^\gamma] \right) &=  \sum_{j=0}^{i-2} \sum_{l \in \Z} p_{\frac{(i-1-j)[n^\theta]}{n}}\left(\frac{l[n^\gamma]-(k+a_{i-1}-a_j)[n^\gamma]}{n^{1/\alpha}}\right)\cdot \sigma\big(u_{j[n^\theta]}\left(l [n^\gamma]\right)\big) \cdot W_j (l) \\
  &= \sum_{j=0}^{i-2} \sum_{l \in \Z} p_{\frac{(i-1-j)[n^\theta]}{n}}\left(\frac{l[n^\gamma]-(k+a_{i-1})[n^\gamma]}{n^{1/\alpha}}\right)\cdot \sigma\big(u_{j[n^\theta]}\left((l-a_j) [n^\gamma]\right)\big) \cdot W_j (l-a_j), 
\end{split}\ee
the second line following from the first by a simple translation. Note that the sum over $j$ runs from $0$ to $i-2$ ; this avoids the singularity of $p$ at $t=0, x=0$.

\begin{proposition}\label{prop:V-Vbar} Fix $T>0$.
We have uniformly for $2\le r\le [Tn^{1-\theta}]$
\bes
\big\|V_{r[n^\theta]}\left( z [n^\gamma]\right) -\bar{V}_{r[n^\theta]}\left( z [n^\gamma]\right) \big\|_{2m}^2 \lesssim \frac{n^{\theta+\gamma+o(1)}}{n^{\min(\alpha-1,\,a)/\alpha}}.
\ees
\end{proposition}
\begin{proof}
 Indeed using the first expression in \eqref{eq:vbar} 
\bes
\begin{split} \label{eq:V-Vbar}
&\left\| V_{r[n^\theta]}\left( z [n^\gamma] \right) -\bar{V}_{r[n^\theta]}\left( z [n^\gamma] \right) \right\|_{2m}^2 \\
& \; \lesssim \frac{n^{\theta+\gamma}}{n^{1-\frac{1}{\alpha}}}
+ \frac{n^{\theta+\gamma}}{n^{1+\frac{1}{\alpha}}}\sum_{j=1}^{r-1} \sum_{l \in \Z} \left\lbrace n^{\frac{1}{\alpha}}\P_{j[n^\theta]} \left(l[n^\gamma]\right) -p_{\frac{j[n^\theta]}{n}}\left(\frac{(l-a_{r-1}+a_{r-1-j})[n^\gamma]}{n^{1/\alpha}}\right)\right \rbrace^2 
\end{split}
\ees
Let us focus on the second term. This is equal to 
 \bes\begin{split} \label{eq:p-P}
 &\frac{n^{\theta+\gamma}}{n^{1+\frac{1}{\alpha}}}\sum_{j=1}^{r-1} \sum_{l \in \Z} \bigg[ n^{\frac{2}{\alpha}}\P^2_{j[n^\theta]} \left(l[n^\gamma]\right) +p^2_{\frac{j[n^\theta]}{n}}\left(\frac{(l-a_{r-1}+a_{r-1-j})[n^\gamma]}{n^{1/\alpha}}\right)\\
 &\hspace{3cm}- 2n^{\frac{1}{\alpha}}\P_{j[n^\theta]} \left(l[n^\gamma]\right)p_{\frac{j[n^\theta]}{n}}\left(\frac{(l-a_{r-1}+a_{r-1-j})[n^\gamma]}{n^{1/\alpha}}\right)\bigg].
\end{split} \ees
 Firstly taking $(1-b)/\alpha=\gamma$ in Theorem \ref{llt} (see Appendix \ref{sec:app:A}) 
 \bes
 \begin{split}
 & \frac{n^{\theta+\gamma}}{n^{1+\frac{1}{\alpha}}}\sum_{j=1}^{r-1} \sum_{l \in \Z}n^{\frac{1}{\alpha}}\P_{j[n^\theta]} \left(l[n^\gamma]\right)  \cdot \left\vert n^{\frac{1}{\alpha}}\P_{j[n^\theta]} \left(l[n^\gamma]\right)-p_{\frac{j[n^\theta]}{n}}\left(\frac{(l-a_{r-1}+a_{r-1-j})[n^\gamma]}{n^{1/\alpha}}\right)\right\vert \\
 &\lesssim  \frac{n^{\theta+\gamma}}{n^{(\alpha-1)/\alpha}}\sum_{j=1}^{r-1} \left[\frac{1}{(jn^\theta)^{(1+a)/\alpha}} +\frac{n^\gamma}{(jn^\theta)^{2/\alpha}}\right] \\
 &\lesssim \frac{n^{\gamma+o(1)}}{n^{\min(\alpha-1,\,a)/\alpha}}+ \frac{n^{2\gamma+o(1)}}{n^{(\alpha-1)/\alpha}}.
 \end{split}
 \ees
 Also we control
\bes
\begin{split}
& \frac{n^{\theta+\gamma}}{n^{1+\frac{1}{\alpha}}}\sum_{j=1}^{r-1} \sum_{l \in \Z} p_{\frac{j[n^\theta]}{n}}\left(\frac{(l-a_{r-1}+a_j)[n^\gamma]}{n^{1/\alpha}}\right)  \cdot \left\vert n^{\frac{1}{\alpha}}\P_{j[n^\theta]} \left(l[n^\gamma]\right)-p_{\frac{j[n^\theta]}{n}}\left(\frac{(l-a_{r-1}+a_j)[n^\gamma]}{n^{1/\alpha}}\right)\right\vert \\
& \lesssim \frac{n^{\theta}}{n^{(\alpha-1)/\alpha}}\sum_{j=1}^{r-1} \left[\frac{1}{(jn^\theta)^{(1+a)/\alpha}} +\frac{n^\gamma}{(jn^\theta)^{2/\alpha}}\right]\sum_{l \in \Z} \frac{n^\gamma}{n^{1/\alpha}}p_{\frac{j[n^\theta]}{n}}\left(l[n^\gamma]n^{-1/\alpha}\right) \\
&\lesssim \frac{n^{\theta}}{n^{(\alpha-1)/\alpha}}\sum_{j=1}^{r-1} \left[\frac{1}{(jn^\theta)^{(1+a)/\alpha}} +\frac{n^\gamma}{(jn^\theta)^{2/\alpha}}\right] \\
&\lesssim \frac{n^{o(1)}}{n^{\min(\alpha-1,\,a)/\alpha}}+ \frac{n^{\gamma+o(1)}}{n^{(\alpha-1)/\alpha}}.
 \end{split}
 \ees
For the Riemann sum approximation of the density $p$ we need $\gamma \le \theta/\alpha$ which we
have assumed in \eqref{g:t:cond}.
\end{proof}

\subsection{Comparison with the stochastic heat equation}
In this last step we compare our approximation $\bar V$ with the stochastic heat equation \eqref{eq:she} with initial profile $0$, with respect to a white noise $\mathcal W$ constructed as 
\bes
\mathcal W_j(l) = W_j(l-a_j).
\ees
Thus $\mathcal W$ is constructed from $W$ by spatially shifting the noise $W$ in 
the time region $\big[j[n^\theta]/n, (j+1)[n^\theta]/n\big)$ 
by $a_j[n^\gamma]$. One can check that $\mathcal W$ is a white noise as follows. For $f \in L^2(\R_+\times \R)$ write
\bes \begin{split} \mathcal W(f) & =\sum_{j=0}^\infty \int_{\frac{j[n^\theta]}{n}}^{\frac{(j+1)[n^\theta]}{n}} \int_{\R} f(t, x)\,\mathcal W(dx,\, ds)\\
&= \sum_{j=0}^\infty  \int_{\frac{j[n^\theta]}{n}}^{\frac{(j+1)[n^\theta]}{n}} \int_{\R} f(t, x+a_j)\, W(dx,\, ds)
\end{split}
\ees
It now follows that $\mathcal W(f)$ is Gaussian with mean $0$ and $\mE[ \mathcal W(f)\cdot \mathcal W(g) ]=(f,g)$ for $f, g \in L^2(\R_+\times \R)$. The unique solution to \eqref{eq:she} with this white noise $\mathcal W$ when the initial profile is $0$ is 
\be \label{stoch_part}
v_t(x) = \int_0^t \int_{\R} p_{t-s}(y-x) \sigma\big(v_s(y)\big) \mathcal W(\d s\,\d y).
\ee
The following H\"older continuity estimate will be very useful for us. 
\begin{lemma}[\cite{foon-khos-09}] \label{holder} For any $p\ge 2$
\bes
\left\|  v_t(x)-v_s(y)\right\|_{p}^2 \lesssim |x-y|^{\alpha-1} +|s-t|^{\frac{\alpha-1}{\alpha}}.
\ees
\end{lemma}
\begin{remark} For general initial profile $v_0$,  the solution $v_t(x)$ is obtained by adding $(p_t*v_0)(x)$ to the right hand side
of \eqref{stoch_part}. Then the conclusion of Lemma \ref{holder} is valid only for the {\it noise term} \eqref{stoch_part}.
\end{remark}
We shall first consider a spatial discretisation of $v$ as in \cite{foon-jose-li} and \cite{jose-khos-muel}, and divide space into intervals of length $[n^\gamma]/n^{1/\alpha}$. Let 
\be\label{eq:tildeV}
\tilde{v}_t\left(\frac{k[n^\gamma]}{n^{1/\alpha}}\right) = \int_0^{t-\frac{[n^\theta]}{n}} \sum_{l \in \Z} p_{t-s} \left(\frac{(l-k)[n^\gamma]}{n^{1/\alpha}}\right) \sigma\left(v_s\left(\frac{l[n^\gamma]}{n^{1/\alpha}}\right)\right)\, \d B_s^{(n)}(l)
\ee
where 
\bes
B_s^{(n)}(l) =\mathcal W\left([0,s]\times \left[\frac{l[n^\gamma]}{n^{1/\alpha}},\,\frac{(l+1)[n^\gamma]}{n^{1/\alpha}}\right)\right).
\ees
One has the following
\begin{lemma} \label{lem:v:s} We have
\be \label{v-tildev}
\left\|v_t(x) - \tilde v_t\left(\frac{[n^\gamma]}{n^{1/\alpha}}\left[\frac{xn^{1/\alpha}}{[n^\gamma]}\right]\right)\right\|_{2m}^2 \lesssim \frac{n^{\theta(\alpha-1)/\alpha}}{n^{(\alpha-1)/\alpha}}.
\ee
\end{lemma}
\begin{proof}In \cite{foon-jose-li} and \cite{jose-khos-muel} we consider the integral in \eqref{eq:tildeV} from $0$ to $t-\frac{n^{\alpha\gamma}}{n}$. Here we are looking at a smaller interval since $\gamma\alpha\le \theta$. If one looks at the arguments in those papers we see we have to control $\int_0^{n^\theta/n} s^{-1/\alpha}\,\d s$ which gives the bound in \eqref{v-tildev}.
\end{proof}
Next we discretise $\tilde v$ in time. Let 
\be \label{eq:barv}
\bar{v}_{\frac{i[n^\theta]}{n}}\left(\frac{k[n^\gamma]}{n^{1/\alpha}}\right)=\sum_{j=0}^{i-2} \sum_{l \in \Z} p_{\frac{(i-1-j)[n^\theta]}{n}}\left(\frac{(l-k)[n^\gamma]}{n^{1/\alpha}}\right)\cdot \sigma\left(v_{\frac{j[n^\theta]}{n}}\left(\frac{l [n^\gamma]}{n^{1/\alpha}}\right)\right) \cdot\mathcal{W}_j (l)
\ee
for $i\ge 2$ and $0$ otherwise, where 
\bes
\mathcal{W}_j(l) = \mathcal{W}\left(\left[\frac{j[n^\theta]}{n}, \,\frac{(j+1)[n^\theta]}{n}\right)\times \left[\frac{l[n^\gamma]}{n^{1/\alpha}},\,\frac{(l+1)[n^\gamma]}{n^{1/\alpha}}\right)\right).
\ees
We will need the following lemma
\begin{lemma}[\cite{kolo}] \label{lem:p:der} We have
\[
\left\vert \frac{\partial}{\partial t} p_t(x)\right\vert \lesssim \frac{p_t(x/2)}{t}.
\]
\end{lemma}
\begin{proof} For $\alpha<2$ one just has to look at estimate (3.7) in \cite{kolo}. In this case
the inequality is true even without the factor of $1/2$ in the right.  For $\alpha=2$
this is just a simple computation involving the Gaussian density.
\end{proof}
The following gives an error bound on the temporal discretisation of $\tilde v$.
\begin{lemma} \label{lem:v:t}  We have for any $k \in \Z$ and $r\in \Z_+$ satisfying \eqref{eq:r:z} 
\[ \left\|\tilde{v}_t\left(\frac{k[n^\gamma]}{n^{1/\alpha}}\right)-\bar{v}_{\frac{r[n^\theta]}{n}}\left(\frac{k[n^\gamma]}{n^{1/\alpha}}\right)  \right\|_{2m}^2  \lesssim  \frac{n^{\theta(\alpha-1)/\alpha}}{n^{(\alpha-1)/\alpha}}. \]
\end{lemma}
\begin{proof}Using the temporal H\"older continuity supplied in Lemma \ref{holder}
\be\begin{split}\label{eq:vbar-vn}
& \left\|\tilde{v}_t\left(\frac{k[n^\gamma]}{n^{1/\alpha}}\right)-\bar{v}_{\frac{r[n^\theta]}{n}}\left(\frac{k[n^\gamma]}{n^{1/\alpha}}\right)  \right\|_{2m}^2 \\
& \lesssim  \left(\frac{n^\theta}{n}\right)^{\frac{\alpha-1}{\alpha}}\int_0^{t-\frac{[n^\theta]}{n}} \d s\,\sum_{l\in \Z} \frac{[n^\gamma]}{n^{1/\alpha}}\cdot p^2_{t-s} \left(\frac{l[n^\gamma]}{n^{1/\alpha}}\right) \\
&\;\; +  \sum_{i=1}^{r} \int_{\frac{i[n^{\theta}]}{n}}^{\frac{(i+1)[n^{\theta}]}{n}} \d s \sum_{l\in \Z} \frac{[n^\gamma]}{n^{1/\alpha}}\left[ p_s\left(\frac{l[n^\gamma]}{n^{1/\alpha}}\right)-p_{\frac{i[n^\theta]}{n}}\left(\frac{l[n^\gamma]}{n^{1/\alpha}}\right)\right]^2
\end{split}\ee
The first term in \eqref{eq:vbar-vn} can be bound by 
\bes\begin{split}
 &\lesssim  \left(\frac{n^\theta}{n}\right)^{\frac{\alpha-1}{\alpha}} \int_{\frac{[n^\theta]}{n}}^t \d s\, \int_{\R} \d x\, p_s^2(x) \\
 &\lesssim \left(\frac{n^\theta}{n}\right)^{\frac{\alpha-1}{\alpha}}\int_{\frac{[n^\theta]}{n}}^t  \frac{\d s}{s^{1/\alpha}} \\
 & \lesssim \frac{n^{\theta(\alpha-1)/\alpha}}{n^{(\alpha-1)/\alpha}}.
\end{split}\ees
Using Lemma \ref{lem:p:der} we can bound the second term in \eqref{eq:vbar-vn} by
\bes
\begin{split}
&\lesssim \sum_{i=1}^r \int_{\frac{i[n^{\theta}]}{n}}^{\frac{(i+1)[n^{\theta}]}{n}} \d s \sum_{l\in \Z} \frac{[n^\gamma]}{n^{1/\alpha}} \left[\int_{\frac{i[n^\theta]}{n}}^s \frac{\d q}{q}\, p_q\left(\frac{l[n^\gamma]}{2n^{1/\alpha}}\right) \right]^2 \\
&\lesssim \frac{n^\theta}{n}\sum_{i=1}^r \int_{\frac{i[n^{\theta}]}{n}}^{\frac{(i+1)[n^{\theta}]}{n}} \d s \sum_{l\in \Z} \frac{[n^\gamma]}{n^{1/\alpha}} \int_{\frac{i[n^\theta]}{n}}^s \frac{\d q}{q^2}\, p^2_q\left(\frac{l[n^\gamma]}{2n^{1/\alpha}}\right) \\
&\lesssim \left(\frac{n^\theta}{n}\right)^2\sum_{i=1}^r \int_{\frac{i[n^{\theta}]}{n}}^{\frac{(i+1)[n^{\theta}]}{n}}   \frac{\d q}{q^2}\,  \sum_{l\in \Z} \frac{[n^\gamma]}{n^{1/\alpha}}\,p^2_q\left(\frac{l[n^\gamma]}{2n^{1/\alpha}}\right) \\
&\lesssim  \left(\frac{n^\theta}{n}\right)^2 \int_{\frac{[n^\theta]}{n}}^t \frac{\d q}{q^{2+\frac1\alpha}}\\
&\lesssim \frac{n^{\theta(\alpha-1)/\alpha}}{n^{(\alpha-1)/\alpha}}.
\end{split}
\ees
We have again used $\gamma \le \theta/\alpha$ for the Riemann sum approximation in removing 
the sum in the third last line. This completes the proof of the lemma.
\end{proof}

We now collect all our estimates in this section to give a bound on the difference between $\bar u_t^{(n)}(x)$ and
$v_t(x)$.
\begin{theorem}
Suppose that $v_0\equiv 0$, and fix $T\ge 0$. We then  have the following uniform bound valid for $2\frac{[n^\theta]}{n}\le t\le T$ 
\be \label{eq:opt} \Big \| \bar u_t^{(n)}(x) - v_t(x) \Big\|_{2m}^2 \lesssim \frac{n^{\gamma}}{n^{(\alpha-1)\theta}}+\frac{n^{\theta+\gamma+o(1)}}{n^{\min(a, \alpha-1)/\alpha}}+\frac{1}{n^{(\theta+\gamma)\cdot \min(1,\,\kappa)/(2m)}}.\ee
\end{theorem} 
\begin{proof}Recall $r, z$ from \eqref{eq:r:z}. Firstly we apply Minkowski's inequality
\be
\begin{split} \label{eq:final}
& \left\| u_{r[n^\theta]}\left(z[n^\gamma] \right) -  {v}_{\frac{r[n^\theta]}{n}}\left(\frac{(z+a_r)[n^\gamma]}{n^{1/\alpha}}\right)\right\|_{2m}^2 \\
&\lesssim  \left\| u_{r[n^\theta]}\left(z[n^\gamma]\right) - \bar V_{r[n^\theta]}(z[n^\gamma])\right\|_{2m}^2 \\
&\;\;+\left\| {v}_{\frac{r[n^\theta]}{n}}\left(\frac{(z+a_{r-1})[n^\gamma]}{n^{1/\alpha}}\right)-{v}_{\frac{r[n^\theta]}{n}}\left(\frac{(z+a_{r})[n^\gamma]}{n^{1/\alpha}}\right)\right\|_{2m}^2\\
&\;\;+\left\| {v}_{\frac{r[n^\theta]}{n}}\left(\frac{(z+a_{r-1})[n^\gamma]}{n^{1/\alpha}}\right)-\bar{v}_{\frac{r[n^\theta]}{n}}\left(\frac{(z+a_{r-1})[n^\gamma]}{n^{1/\alpha}}\right)\right\|_{2m}^2\\
&\;\; + \left\|\bar V_{r[n^\theta]}(z[n^\gamma])-\bar{v}_{\frac{r[n^\theta]}{n}}\left(\frac{(z+a_{r-1})[n^\gamma]}{n^{1/\alpha}}\right)\right\|_{2m}^2 
\end{split}
\ee
By Lemmas \ref{lem:v:s} and \ref{lem:v:t} of this subsection as well as Propositions \ref{prop:u-U}, \ref{prop:V-U}, \ref{prop:V-Vbar} we can bound the above by
\be\begin{split} \label{eq:gronwall}
&\lesssim \frac{n^{\gamma}}{n^{(\alpha-1)\theta}}+\frac{n^{\theta+\gamma+o(1)}}{n^{\min(a, \alpha-1)/\alpha}}+ \frac{1}{n^{(\theta+\gamma)\cdot \min(1,\,\kappa)/(2m)}}\\
&\;\;+ \frac{n^{\gamma+\theta}}{n^{1+\frac1\alpha}}\sum_{j=0}^{r-2}\sum_{l\in \Z}p^2_{\frac{(r-1-j)[n^\theta]}{n}}\left(\frac{(l-k)[n^\gamma]}{n^{1/\alpha}}\right)\cdot \left\|{v}_{\frac{j[n^\theta]}{n}}\left(\frac{l [n^\gamma]}{n^{1/\alpha}}\right)-u_{j[n^\theta]}\left((l-a_j)[n^\gamma]\right)\right\|_{2m}^2.
\end{split}
\ee
An application of Gronwall's inequality then gives
\bes
 \left\| u_{r[n^\theta]}\left(z[n^\gamma] \right) -  {v}_{\frac{r[n^\theta]}{n}}\left(\frac{(z+a_r)[n^\gamma]}{n^{1/\alpha}}\right)\right\|_{2m}^2 \lesssim \frac{n^{\gamma}}{n^{(\alpha-1)\theta}}+\frac{n^{\theta+\gamma+o(1)}}{n^{\min(a, \alpha-1)/\alpha}}+ \frac{1}{n^{(\theta+\gamma)\cdot \min(1,\,\kappa)/(2m)}}.
\ees
We can now use the H\"older continuity esitmates in Lemma \ref{holder} and Proposition \ref{prop:u:hol} to conclude the proof.
\end{proof}

We now complete the proof of Theorem \ref{thm:main}.
\begin{proof}[Proof of Theorem \ref{thm:main}]
The case of $v_0\equiv 0$ has already been covered. For general $v_0$, the solution to \eqref{eq:she} with white noise $\mathcal W$ is 
\[ v_t(x) = (p_t*v_0)(x) +\int_0^t \int_{\R} p_{t-s}(y-x) \sigma\big(v_s(y)\big) \mathcal W(\d s\,\d y).\]
We have already shown in \eqref{eq:opt} that the difference of the noise terms above and in \eqref{eq:dis:sc:sol} goes to $0$ in $\|\cdot \|_{2m}$ norm.  It is thus enough to show that the  difference of the $\|\cdot\|_{2m}$ of the non-noise terms goes to $0$, or
\[ \left \| E_{z[n^\gamma]} \, v_0\left(\frac{X_{r[n^\theta]}}{n^{1/\alpha}}\right) - \Big(p_{\frac{r[n^\theta]}{n}}* v_0\Big)\left(\frac{(z+a_r)[n^\gamma]}{n^{1/\alpha}}\right)\right\|_{2m}^2 \rightarrow 0.\]
To see this first note that the expression inside $\|\cdot \|_{2m}$ goes
to $0$ almost surely by the weak convergence of the centered $X_n$ to a Stable($\alpha$) random variable,
and by the smoothing properties of the Stable($\alpha$) density. The dominated convergence theorem can then
be applied to show the above. This then implies
\[ \left\| u_{r[n^\theta]}\left(z[n^\gamma] \right) -  {v}_{\frac{r[n^\theta]}{n}}\left(\frac{(z+a_r)[n^\gamma]}{n^{1/\alpha}}\right)\right\|_{2m}^2\to 0\]
and the conclusion of the theorem follows from Lemma \ref{holder} and Proposition \ref{prop:u:hol}.
\end{proof}

\section{Proof of Theorem \ref{thm:char}} \label{sec:char}
In order to couple the initial profile $\eta$ with a Brownian motion we need the following proposition 
whose proof is similar to that of Theorem \ref{thm:couple}. For $0<\theta'<1/2$ let
\bes t_l=\frac{l[n^{\theta'}]}{n},\qquad \bar \zeta_l = \sum_{k=(l-1)[n^{\theta'}]}^{l[n^{\theta'}]-1} \frac{\eta(k)}{\sqrt{n}}.
\ees
We have
\begin{proposition} [\cite{kana}]\label{prop:couple:2}  For each $n\in \N$, one can construct a copy of $\eta$ and a two sided Brownian motion $B$ with variance $\lambda$ on a probability space, so that for $m\in \N$ with $2m <2+\kappa'$
\be \label{eq:couple:2}
\mE\left[\left \vert  \bar \zeta_l - \left(B_{t_l}-B_{t_{l-1}}\right)\right \vert^{2m}\right] \lesssim n^{-m(1-\theta')}n^{-\theta'\cdot\frac{\min(1,\kappa')}{2}}.
\ee
\end{proposition}
We can now provide the 
\begin{proof}[Proof of Theorem \ref{thm:char}] A look at the proof of Theorem \ref{thm:u} tells us that existence and uniqueness of $u=u^{(n)}$ holds. However the bounds in \eqref{eq:mom:unif} and \eqref{eq:thm:um} do not hold uniformly in $i\le [nT], k \in \Z$, although it is easily checked that it continues to hold for the noise term. Let us show instead that the $\|\cdot\|_{2m}$ norm of the non-noise
term is finite. We split
\[ \sum_l \P_{i+1}(k, l)\cdot  u_0(l) = \sum_l \P_{i+1}(k, l) \cdot \big[u_0(l) - u_0(k+[i\mu]) + u_0(k+[i\mu])\big].\]
Therefore
\[\big\| \sum_l \P_{i+1}(k, l)\cdot u_0(l) \big\|_{2m}^2 \lesssim \big\| u_0(k +[i\mu])\big\|_{2m}^2 + \sum_l \P_{i+1}(k, l) \cdot \frac{|l-k-[i\mu]|}{\sqrt n},\]
which is finite thanks to our assumptions \eqref{eta} and Lemma \ref{lem:xunif}.

A careful reader of the proof of Theorem \ref{thm:main} would have observed that our assumption 
of bounded initial profile was only needed so that we had
\bes \sup_{j \le [nT],\, l \in \Z }\left\| \sigma\big(u_j(l)\big)\right\|_{2m} < \infty. \ees
While this does not hold for general Lipschitz $\sigma$ if $u_0$ is unbounded in $L^{2m}$, it clearly holds in the case $\sigma(x)$ is bounded uniformly in $x$. Therefore to complete the proof of Theorem \ref{thm:char}
we just need to approximate the non-noise term. As before we divide the proof into several steps. \\
\textbf{Step 1:} First we collect the initial random variables $n^{-1/4}\eta$ into groups of size $[n^{\theta'}]$. Thus 
\bes\begin{split} \label{step1}
&\sum_l \P_{[nt]}\left([x\sqrt n]-[\mu nt],\,l\right) \cdot\Big\lbrace n^{-1/4}\sum_{i=1}^l \eta(i)\Big\rbrace\\
& =\sum_k \sum_{l=k[n^{\theta'}]}^{(k+1)[n^{\theta'}]-1} \P_{[nt]}\left([x\sqrt n]-[\mu nt],\,l\right)  \cdot\Big\lbrace n^{-1/4}\sum_{i=1}^{k[n^{\theta'}]} \eta(i)\Big\rbrace \\
&\;\;+ \sum_k  \bigg[\sum_{l=k[n^{\theta'}]}^{(k+1)[n^{\theta'}]-1} \P_{[nt]}\left([x\sqrt n]-[\mu nt],\,l\right)  \cdot\bigg\lbrace n^{-1/4}\sum_{i=k[n^{\theta'}]+1}^{l} \eta(i)\bigg\rbrace\bigg] .
\end{split}
\ees
Consider the second term in the right hand side. The expression in square brackets is independent of $k$. 
Thus using Burkholder's inequality one gets the following bound for $\|\cdot\|_{2m}^2$ of the second term
\bes
\begin{split}\label{step1:2}
&\lesssim \frac{1}{\sqrt n} \sum_k \sum_{l,l'=k[n^{\theta'}]}^{(k+1)[n^{\theta'}]-1}\P_{[nt]}\left([x\sqrt n]-[\mu nt],\,l\right) \cdot \P_{[nt]}\left([x\sqrt n]-[\mu nt],\,l'\right) \\
&\hspace{4cm}\times\bigg\|\sum_{i=k[n^{\theta'}]+1}^l \eta(i)\cdot \sum_{j=k[n^\theta]+1}^{l'} \eta(j) \bigg\|_m \\
&\lesssim \frac{n^{\theta'}}{\sqrt n} P \left(\left\vert X_{[nt]}-\tilde X_{[nt]} \right \vert \le n^{\theta'}\right),
\end{split}
\ees
where one uses the Cauchy-Schwarz inequality to obtain the last step. This goes to $0$ as $n\to \infty$ since $\theta'<1/2$. \\
\textbf{Step 2:}  Next we replace the sum over $\eta$ in blocks by the Brownian motion constructed in Proposition \ref{prop:couple:2}.
\bes\begin{split}
&\sum_k \sum_{l=k[n^\theta]}^{(k+1)[n^{\theta'}]-1} \P_{[nt]}\left([x\sqrt n]-[\mu nt],\, l\right)  \cdot \bigg\lbrace n^{-1/4} \sum_{i=1}^{k[n^{\theta'}]} \eta(i)\bigg\rbrace \\
&= \sum_{k,l} n^{1/4} \P_{[nt]}\left([x\sqrt n]-[\mu nt],\, l\right) \cdot  B_{\frac{k[n^{\theta'}]}{n}} \\
&\;\; +\sum_{k,l} n^{1/4} \P_{[nt]}\left([x\sqrt n]-[\mu nt],\, l\right) \bigg[\sum_{i=1}^k\bar \zeta_i- B_{\frac{k[n^{\theta'}]}{n}}\bigg]
\end{split}\ees
where the limits of the sum are as in the first line. Using the Cauchy-Schwarz inequality we bound $\|\cdot \|_{2m}^2$ of the second term as follows 
\bes
\begin{split}
& \lesssim \sqrt n \bigg\lbrace \sum_{k,l}  \P_{[nt]}\left([x\sqrt n]-[\mu nt],\, l\right)\cdot \bigg\|\sum_{i=1}^k \bar\zeta_i- B_{\frac{k[n^{\theta'}]}{n}}\bigg\|_{2m}\bigg\rbrace^2 \\
& \lesssim \sqrt n \sum_{k, l} \P_{[nt]}\left([x\sqrt n]-[\mu nt],\, l\right) \cdot \bigg\|\sum_{i=1}^k \bar\zeta_i- B_{\frac{k[n^{\theta'}]}{n}}\bigg\|_{2m}^2. 
\end{split}
\ees
We now use \eqref{eq:couple:2} to bound this
\bes
\begin{split}
& \lesssim \sqrt n \sum_{k, l} \P_{[nt]}\left([x\sqrt n]-[\mu nt],\, l\right) \cdot |k| \cdot \frac{n^{\theta'(1-\frac{\min(1,\kappa')}{2m})}}{n} \\
& \lesssim \frac{1}{n^{\frac{\theta'\min(1,\kappa')}{2m}}} \cdot E_{[x\sqrt n]} \left(\frac{\left|X_{[nt]}- [\mu nt]\right|}{\sqrt n}\right) \\
& \lesssim \frac{1}{n^{\frac{\theta'\min(1,\kappa')}{2m}}},
\end{split}
\ees
where we used Lemma \ref{lem:xunif} in the last line. The above also goes to $0$ with $n$.\\
\textbf{Step 3:} As in the proof of Theorem \ref{thm:main} we now substitute the transition probability of the random walk by the heat kernel.
\bes
\begin{split}
& \sum_k \sum_{l=k[n^{\theta'}]}^{(k+1)[n^{\theta'}]-1}  n^{1/4} \,\P_{[nt]}\left([x\sqrt n]-[\mu nt],\, l\right) \cdot  B_{\frac{k[n^{\theta'}]}{n} } \\
& =[n^{\theta'}]\sum_k  p_t \left(\frac{k[n^{\theta'}]-[x\sqrt n]}{\sqrt n}\right)  \cdot  n^{-1/4} B_{\frac{k[n^{\theta'}]}{n}} \\
&\;\; +   \sum_k \sum_{l=k[n^{\theta'}]}^{(k+1)[n^{\theta'}]-1}  \left[\sqrt n \,\P_{[nt]}\left([x\sqrt n]-[\mu nt],\, l\right)  - p_t \left(\frac{k[n^{\theta'}]-[x\sqrt n]}{\sqrt n}\right) \right]\cdot n^{-1/4} B_{\frac{k[n^{\theta'}]}{n}}.
\end{split}
\ees
We now bound $\|\cdot \|_{2m}^2$ of the second term on the right hand side. This gives an upper bound of 
\be
\begin{split} \label{step3}
& \frac{1}{\sqrt n} \bigg\lbrace  \sum_{k, l} \bigg \vert \sqrt n \,\P_{[nt]}\left([x\sqrt n]-[\mu nt],\, l\right)  - p_t \bigg(\frac{k[n^{\theta'}]-[x\sqrt n]}{\sqrt n}\bigg) \bigg \vert \cdot \bigg\| B_{\frac{k[n^{\theta'}]}{n}}\bigg\|_{2m}\bigg \rbrace^2  \\
& \lesssim \sum_{k, l}\bigg\vert \sqrt n \,\P_{[nt]}\left([x\sqrt n]-[\mu nt],\, l\right)  - p_t \bigg(\frac{k[n^{\theta'}]-[x\sqrt n]}{\sqrt n}\bigg) \bigg \vert \cdot \frac{|k|n^{\theta'}}{n}
\end{split}
\ee
where we used the Cauchy-Schwarz inequality along with the bound $\big \| B_{\frac{k[n^{\theta'}]}{n}}\big\|_{2m}^2 \lesssim \frac{kn^{\theta'}}{n}$. Fix $\epsilon>0$ such that $2\epsilon < \min(a/2,\, (1-2\theta')/2)$. Continuing with our bound \eqref{step3} we use Theorem \ref{llt} to get 
\bes
\begin{split}
&\lesssim \sum_k \sum_{|l - [x\sqrt n]| \le n^{\frac12+\epsilon}} \frac{|k|n^{\theta'}}{n}\cdot \left[\frac{1}{n^{a/2}}+ \frac{1}{n^{(1-2\theta')/2}}\right] \\
&\quad + \sum_k \sum_{|l-[x\sqrt n]|>n^{\frac12+\epsilon}} \frac{|k|n^{\theta'}}{n}\cdot \left[\sqrt n \, \P_{[nt]}\left([x\sqrt n] -[\mu nt],\, l\right) + p_t \left(\frac{k[n^{\theta'}]-[x\sqrt n]}{\sqrt n}\right)\right] \\
& \lesssim \frac{n^{1+2\epsilon}}{n}\cdot \left[\frac{1}{n^{a/2}}+ \frac{1}{n^{(1-2\theta')/2}}\right]  + \sum_{k:\, |kn^{\theta'} -x\sqrt n| >n^{\frac12+\epsilon}} \frac{n^{\theta'}}{\sqrt n}\cdot \frac{|k|n^{\theta'}}{\sqrt n} \cdot  p_t \left(\frac{k[n^{\theta'}]-[x\sqrt n]}{\sqrt n}\right)\\
&\quad +  \frac{1}{\sqrt n} E\left[\left(\big\vert X_{[nt]}-\mu [nt]\big \vert + n^{\theta'}\right)\cdot \ind\big\lbrace |X_{[nt]}-\mu [nt]|>n^{\frac12+\epsilon} \big\rbrace\right]   
\end{split}
\ees
A bound for the third term comes from
applying H\"older's inequality along with Chebyshev's inequality:
\bes
\begin{split}&
\frac{1}{\sqrt{n}} E\left[\vert X_{[nt]}-\mu nt\big \vert \cdot \ind\big\lbrace |X_{[nt]}-\mu nt|>n^{\frac12+\epsilon} \big\rbrace\right]\\
&\le \frac{1}{\sqrt{n}} E\left[\vert X_{[nt]}-\mu nt\big \vert^2\right]^{1/2}\cdot P\big[ |X_{[nt]}-\mu nt|>n^{\frac12+\epsilon} \big]^{1/2}\\
& \lesssim n^{-\epsilon},
\end{split}
\ees
while the second term decays much more quickly because of the exponential decay of the heat kernel.\\
\textbf{Step 4:} To complete the proof of the theorem it is enough for us to show 
\bes
[n^{\theta'}]\sum_k p_t \left(\frac{k[n^{\theta'}]-[x\sqrt n]}{\sqrt n}\right)  \cdot  \frac{B_{\frac{k[n^{\theta'}]}{n}}}{n^{1/4} }\approx \int_{\R} p_t(x-y) \tilde{B}(y) \, \d y
\ees
where 
\bes
\tilde{B}(y) = n^{\frac14} B_{\frac{y}{\sqrt n}}
\ees
is another Brownian motion. As the reader might have guessed this is just a Riemann sum argument. We just need
to bound 
\be \begin{split} \label{eq:4}
 &\bigg \| \sum_k \int_{\frac{k[n^{\theta'}]}{\sqrt n}}^{\frac{(k+1)[n^{\theta'}]}{\sqrt n}} p_t(y-x) \cdot \bigg\lbrace \tilde B(y)-\tilde B\bigg(\frac{k[n^{\theta'}]}{\sqrt n}\bigg)\bigg\rbrace \d y \bigg\|_{2m}^2 \\
& \; + \bigg\|\sum_k\int_{\frac{k[n^{\theta'}]}{\sqrt n}}^{\frac{(k+1)[n^{\theta'}]}{\sqrt n}} \bigg\lbrace p_t(y-x) -p_t\bigg(\frac{k[n^{\theta'}]-[x\sqrt n]}{\sqrt n}\bigg)\bigg\rbrace\cdot \tilde B\bigg(\frac{k[n^{\theta'}]}{\sqrt n} \bigg) \d y  \bigg\|_{2m}^2.
\end{split}\ee
For the first term we use the independence of Brownian increments as well as Burkholder's inequality to get a 
bound of 
\bes \begin{split}
 & \bigg\| \sum_k \bigg[ \int_{\frac{k[n^{\theta'}]}{\sqrt n}}^{\frac{(k+1)[n^{\theta'}]}{\sqrt n}} p_t(y-x) \bigg\lbrace \tilde B(y)-\tilde B\bigg(\frac{k[n^{\theta'}]}{\sqrt n}\bigg)\bigg\rbrace \,\d y\bigg]^2\bigg\|_{m} \\
& \lesssim \frac{n^{\theta'}}{\sqrt n}\sum_k \int_{\frac{k[n^{\theta'}]}{\sqrt n}}^{\frac{(k+1)[n^{\theta'}]}{\sqrt n}} p_t(y-x) \bigg \| \tilde B(y)-\tilde B\bigg(\frac{k[n^{\theta'}]}{\sqrt n}\bigg) \bigg\|_{2m}^2 \,\d y \\
& \lesssim \frac{n^{2\theta'}}{ n}.
\end{split}\ees
Next we have the following bound on $\|\cdot\|_{2m}^2$ for the second term in \eqref{eq:4}
\bes
\begin{split}
& \lesssim \bigg[\sum_k\int_{\frac{k[n^{\theta'}]}{\sqrt n}}^{\frac{(k+1)[n^{\theta'}]}{\sqrt n}} \bigg| p_t(y-x) -p_t\bigg(\frac{k[n^{\theta'}]-[x\sqrt n]}{\sqrt n}\bigg)\bigg|\cdot \bigg \|\tilde B\bigg(\frac{k[n^{\theta'}]}{\sqrt n} \bigg)\bigg\|_{2m}\d y \bigg]^2 \\
& \lesssim  \bigg[\sum_k\int_{\frac{k[n^{\theta'}]}{\sqrt n}}^{\frac{(k+1)[n^{\theta'}]}{\sqrt n}} \bigg| p_t(y-x) -p_t\bigg(\frac{k[n^{\theta'}]-[x\sqrt n]}{\sqrt n}\bigg)\bigg|\cdot \frac{n^{{\theta'}/2}\sqrt{|k|}}{n^{1/4}} \, \d y \bigg]^2\\
& \lesssim \frac{1}{t} \cdot \frac{n^{2\theta'}}{ n}  \left[\int p_t\big((y-x)/2\big) \sqrt{y}\, \d y \right]^2\\
& \lesssim  \frac{n^{2\theta'}}{n}.
\end{split}
\ees
 We have shown that \eqref{eq:4} goes to $0$ with $n$ and this completes the proof of Theorem \ref{thm:char}.
\end{proof}

\section{Tightness} \label{sec:tight}
In this section we discuss the issue of tightness in Theorem \ref{thm:main}. 
\begin{theorem} \label{thm:tight} Suppose the initial profile $v_0$ satisfies 
\[ \left\| v_0(x)- v_0(y) \right\|_{2m} \lesssim |x-y|^\zeta\]
for some $0<\zeta<\alpha/2$, all $x, y\in \R$ and some 
\be \label{eq:m} m> \frac{\max(2\alpha, \alpha+1)}{\min (2\zeta, \alpha-1)}.\ee
Then the process $\bar u_t(x), \, t\in \R_+, \, x\in \R$ is
tight in every compact subset of $\R_+\times \R$.
\end{theorem}
\begin{remark} One can check that the following tightness arguments will hold for Theorem \ref{thm:char} with $\zeta=1/2$.
\end{remark}
We show tightness for the process $\bar u_t(x) =\bar u_t^{(n)}(x)$ in the time-space 
box $[0,1]\times [0,1]$, although the argument is valid for any compact subset of $\R_+ \times \R$. Define the modulus of continuity
\bes
w_{\delta}(\bar u) = \sup_{\substack{(t,x),\, (s,y) \in [0,1]^2,\\  |(t,x)-(s,y)|<\delta}}  \left\vert \bar u_t(x) - \bar u_s(y) \right\vert.
\ees
Tightness of the process $\bar u$ will follow from the following
\begin{lemma}[\cite{kuma}] \label{lem:tight} Suppose there is a sequence $\delta_n \downarrow 0$ such that the following hold. 
\begin{enumerate}
\item There is $\psi>0$, $\lambda>2$ and a constant $c_1$ such that for all large enough $n$
\bes
\mE \left ( \left \vert \bar u_t(x) -\bar u_s(y) \right \vert^\psi \right) \le c_1 \left| (t,x) -(s,y)\right|^\lambda,
\ees
for all $(t,x), (s,y) \in [0,1]^2$ and $\left| (t,x) -(s,y)\right|> \delta_n$.
\item For all $\epsilon, \rho>0$,
\bes
\mP\left[w_{\delta_n}(\bar u) > \epsilon\right]< \rho
\ees
for all large $n$.
\end{enumerate}
Then for all $\epsilon, \, \rho>0$ there is a $0<\delta<1$ such that 
\bes
\mP\left[w_{\delta}(\bar u) > \epsilon\right]< \rho
\ees
for all large $n$.
\end{lemma}

\begin{proof}[Proof of Theorem \ref{thm:tight}]
Let us check the first condition in Lemma \ref{lem:tight}. We use the triangle inequality
\be \label{kolm} \big\| \bar u_t(x) - \bar u_s(y)\big\|_{2m}^2 \le 2\big\| \bar u_t(x) - \bar u_t(y)\big\|_{2m}^2 + 2 \big\| \bar u_t(y) - \bar u_s(y)\big\|_{2m}^2 \ee
and bound each of the two terms on the right, starting with the first. Recall from \eqref{eq:dis:sc:sol} that the solution $u$
is the sum of two terms, and it is 
enough to look at each of them separately. The contribution of the noise term, the one involving $\xi$, gives a bound 
\bes\begin{split}
& \lesssim \frac{1}{n^{(\alpha-1)/\alpha}} \sum_{j=0}^{[nt]-1} \sum_k \left[\P_j(k-[xn^{1/\alpha}])-\P_j(k-[yn^{1/\alpha}])\right]^2 \\
&\lesssim \frac{1}{n^{(\alpha-1)/\alpha}}\, \bar a\left([xn^{1/\alpha}]-[yn^{1/\alpha}])\right)  \\
&\lesssim |x-y|^{\alpha-1} + \frac{1}{n^{(\alpha-1)/\alpha}}.
\end{split}\ees
The contribution of the non-noise term in \eqref{eq:dis:sc:sol} corresponding to the initial profile is 
\bes \begin{split}
&\lesssim \left\| \sum_l \P_{[nt]}(l+[\mu nt]) \cdot \left[u_0\big(l+[xn^{1/\alpha}]\big)-u_0\big(l+[yn^{1/\alpha}]\big)\right]\right\|_{2m}^2 \\
&\lesssim |x-y|^{2\zeta} +\frac{1}{n^{2\zeta/\alpha}}.
\end{split}\ees
We have obtained 
\be \label{kolm:space} \big\| \bar u_t(x) - \bar u_t(y)\big\|_{2m}^2 \lesssim |x-y|^{\alpha-1} + \frac{1}{n^{(\alpha-1)/\alpha}}+|x-y|^{2\zeta} +\frac{1}{n^{2\zeta/\alpha}}. \ee
Let us next consider the second term in \eqref{kolm}. By arguments similar to those in Proposition \ref{prop:u:hol} one can see that the contribution from the noise term 
\bes\begin{split}
&\lesssim \frac{1}{n^{(\alpha-1)/\alpha}} \sum_{j=[ns]}^{[nt]-1} P(Y_j=0)\\
&\qquad  + \frac{1}{n^{(\alpha-1)/\alpha}} \sum_{j=0}^{[ns]-1} \sum_k \left[\P_{[nt]-j-1}(l+[\mu nt]) -\P_{[ns]-j-1}(l+[\mu ns])\right]^2 \\
&\lesssim \left(\frac{[nt]-[ns]}{n}\right)^{(\alpha-1)/\alpha} + \frac{1}{n^{(\alpha-1)/\alpha}} \Big[E\, \bar a\left(W_{[nt]-[ns]}+[\mu nt]-[\mu ns]\right) + E\,\bar a\left(Z_{[nt]-[ns]}\right)\Big] \\
&\lesssim |t-s|^{(\alpha-1)/\alpha} + \frac{1}{n^{(\alpha-1)/\alpha}}.
\end{split}\ees
The contribution of the non-noise term
\bes\begin{split}
&\lesssim \Big[\sum_l \P_{[nt]-[ns]}\big([yn^{1/\alpha}]-[\mu nt],\, l\big) \cdot \big\| u_{[ns]}\big(l\big)-u_{[ns]}\big([yn^{1/\alpha}]-[\mu ns]\big) \big\|_{2m} \Big]^2 \\
& \lesssim \sum_l \P_{[nt]-[ns]}\big([yn^{1/\alpha}]-[\mu nt],\, l\big) \cdot \big\| u_{[ns]}\big(l\big)-u_{[ns]}\big([yn^{1/\alpha}]-[\mu ns]\big) \big\|_{2m}^2 \\ 
& \lesssim \frac{1}{n^{(\alpha-1)/\alpha}} + \frac{1}{n^{2\zeta/\alpha}} + E \Big|X_{[nt]-[ns]}-[\mu nt]+[\mu ns]\Big|^{\alpha-1} +E \Big|X_{[nt]-[ns]}-[\mu nt]+[\mu ns]\Big|^{2\zeta}\\
& \lesssim |t-s|^{(\alpha-1)/\alpha} + \frac{1}{n^{(\alpha-1)/\alpha}}+|t-s|^{2\zeta/\alpha} +\frac{1}{n^{2\zeta/\alpha}}.
\end{split}\ees
The last line is because \eqref{eq:xunif} also holds with $\alpha-1$ replaced by $2\zeta$; this can be seen from following the proof of Lemma \ref{lem:xunif}. We have obtained 
\be \label{kolm:time} \big\| \bar u_t(y) - \bar u_s(y)\big\|_{2m}^2 \lesssim |t-s|^{(\alpha-1)/\alpha} + \frac{1}{n^{(\alpha-1)/\alpha}}+|t-s|^{2\zeta/\alpha} +\frac{1}{n^{2\zeta/\alpha}}. \ee
Combining \eqref{kolm:space} and \eqref{kolm:time} we obtain
\[\mE\big[\big\vert \bar u_t(x) -\bar u_s(y)\big\vert^{2m}\big] \le \big\vert (t,x)-(s,y)\big\vert^{m\cdot \min [2\zeta,\,\alpha-1]/\alpha} \]
when $|(t,x)-(s,y)|> \delta_n=n^{-1}$.  Because of our condition on $m$ the exponent is greater than $2$ and this verifies condition 1 in Lemma \ref{lem:tight}.

We next check the second condition in Lemma \ref{lem:tight}.  Since we have chosen $\delta_n=n^{-1}$ we just need to consider the maximum over $u_i(k)-u_j(l)$ where $|i-j|\le 1,\, |k-l|\le 1$.
There are of the order $n\times n^{1/\alpha}$ such points. Thus
\bes\begin{split}
\mP\left[w_{n^{-1}}(\bar u)>\epsilon\right] & \lesssim n^{1+\frac1\alpha}\sup_{|i-j|\le 1,\, |k-l|\le 1} \mP\left[\left|u_i(k)-u_j(l)\right|>\epsilon\right] \\
& \lesssim n^{1+\frac1\alpha}\, \sup_{|i-j|\le 1,\, |k-l|\le 1} \frac{\mE\left[\left|u_i(k)-u_j(l)\right|^{2m}\right]}{\epsilon^{2m}} \\
&\lesssim \frac{ n^{1+\frac1\alpha}}{\epsilon^{2m}} \cdot \frac{1}{n^{m\cdot \min [2\zeta ,\alpha-1]/\alpha}},
\end{split}\ees
where we used \eqref{kolm:space} and \eqref{kolm:time} for the last line. The above goes to $0$ with $n$ due to our restriction on $m$. Condition 2 of Lemma \ref{lem:tight} is verified.
\end{proof}

\section{Proof of part 1 of Theorem \ref{thm:part}} \label{sec:part:1}
Consider the martingale $M_n$ in \eqref{part:sc} (we ignore the superscript $\tilde \xi$). By the independence of the environment
this has the same distribution as the 
 \be  \label{part:rev}
M^{\leftarrow}_{(n,k)}=\frac{E_k \exp \left( \sum_{i=0}^n \beta\cdot \tilde\xi_{n-i}(X_i)\right)}{\left[\bE e^{ \beta\tilde\xi}\right]^{n+1}},
\ee
where $\E_k$ is the expectation of paths of the random walk starting at $k$. It is therefore enough to prove Theorem \ref{thm:part} for $M^{\leftarrow}_{(n,k)}$. Consider the random field 
\be \label{eq:w}
w_i(k)= \frac{E_k \exp \left( \sum_{j=0}^i \beta \cdot \tilde\xi_{i-j}(X_j)\right)}{\left[\bE e^{ \beta\tilde\xi}\right]^{i+1}},\quad i \ge 0,\, k \in \Z.
\ee
Using the Markov property we obtain 
\be \label{eq:w:she} 
\begin{split}
w_{i+1}(k)& = \sum_l \P(k,l)\, w_i(l) \left(\frac{e^{ \beta\tilde\xi_{i+1}(k)}}{\bE e^{ \beta \tilde\xi}}\right) \\
& = \sum_l \P(k,l)\, w_i(l) +\sum_l \P(k,l)\, w_i(l) \left(\frac{e^{ \beta \tilde\xi_{i+1}(k)}}{\bE e^{ \beta  \tilde\xi}}-1\right).
\end{split}
\ee
with initial profile 
\[ w_0(k) = \frac{e^{\beta \tilde\xi_0(k) }}{\mathbb E e^{\beta \tilde\xi}}.\]
Although \eqref{eq:w:she} is not of the form \eqref{eq:dis:sc} one can check that the statements of Theorem 
\ref{thm:u} hold for $w$; to see this one could follow the arguments of the proof of Theorem \ref{thm:u}. Indeed a solution to 
\eqref{eq:w:she} is 
\[ w_{i+1}(k)= \sum_l \P_{i+1}(k, l) w_0(l) + \sum_{j=0}^i \sum_l \P_{i-j}(k, l) \sum_y \P(l, y) w_j(y) \cdot \Big[\frac{e^{\beta \tilde \xi_{j+1}(y)}}{\bE e^{\beta \tilde \xi}}-1\Big]. \]
Therefore by Burkholder's inequality
\be
\begin{split}
&  \big\|w_{i+1}(k) \big\|_{2m}^2 \\
&\lesssim \sum_l \P_{i+1}(k, l) \cdot \big\|w_0(l) \big\|_{2m}^2 +\sum_{j=0}^i \sum_l \P_{i-j}^2(k, l) \cdot \Big\| \sum_y \P(l, y) w_j(y) \cdot \Big[\frac{e^{\beta \tilde \xi_{j+1}(y)}}{\bE e^{\beta \tilde \xi}}-1\Big]\Big\|_{2m}^2 \\
&\lesssim \sum_l \P_{i+1}(k, l) \cdot \big\|w_0(l) \big\|_{2m}^2 +\sum_{j=0}^i \sum_l \P_{i-j}^2(k, l) \frac{\sup_y \big\| w_j(y)\big\|_{2m}^2}{n^{(\alpha-1)/\alpha}},
\end{split}
\ee
thanks to the independence of $w_j(y)$ and $\tilde \xi_{j+1}(y)$, and the first bound in \eqref{eq:exp-lin}. This argument shows that we obtain the same moment bounds for $w$ as in Theorem \ref{thm:u}.

To obtain the first part of Theorem \ref{thm:part} we shall show that our random field $w$ is {\it close} to the solution $u$ of
\be \begin{split} \label{eq:u}
u_{i+1}(k) &= \sum_l \P(k, l) \,u_i(l) + \beta \tilde \xi_{i+1}(k) \cdot u_i(k) \\
u_0(k) &=1.
\end{split}
\ee
We start with an estimate of the spatial H\"older continuity of $u$.
\begin{lemma} \label{lem:u:k-l}The following holds
\[\sup_{i\le n} \big\| u_i(k)-u_i(l)\big\|_{2m}^2 \lesssim \frac{|l-k|^{\alpha-1}}{n^{(\alpha-1)/\alpha}}. \]
\end{lemma}
\begin{proof} From \eqref{eq:thm:um} the $2m$'th moment of $u_i(k)$ is bounded for $i \le n,\, k \in \Z$.
Therefore one can argue
\bes\begin{split} \big\| u_i(k)-u_i(l)\big\|_{2m}^2 &\lesssim \frac{1}{n^{(\alpha-1)/\alpha}} \sum_{j=0}^{i-1} \sum_y \left(\P_{i-1-j} (k,y) -\P_{i-1-j}(l,y)\right)^2 \\
& \lesssim \frac{\bar a(l-k)}{n^{(\alpha-1)/\alpha}}.
\end{split}\ees
The result follows from Lemma \ref{pker:bd}.
\end{proof}
Consider now the random field $\tilde w$ defined as 
\be \label{eq:wtilde}
\tilde w_{i+1}(k) = \sum_l \P(k,l)\, \tilde w_i(l) +\sum_l \P(k,l)\, \tilde w_i(l) \cdot \beta \tilde\xi_{i+1}(k)
\ee 
with initial profile $1$. The statements of Theorem \ref{thm:u} hold for $\tilde w$ also. We first compare $\tilde w$ with $u$. 
\begin{lemma} The following holds
\bes
\sup_{i\le n, \, k \in \Z} \big\| u_i(k)-\tilde w_i(k)\big\|_{2m}^2 \lesssim \frac{1}{n^{(\alpha-1)/\alpha}}.
\ees
\end{lemma}
\begin{proof}
We can check that 
\bes
\tilde w_{i+1}(k) - u_{i+1}(k) = \beta \sum_{j=0}^i \sum_l \P_{i-j}(k,\, l) \cdot \tilde\xi_{j+1}(l) \sum_{y} \P(l, \, y) \big[\tilde w_j(y) - u_j(l)\big].
\ees
Split $\tilde w_j(y)- u_j(l)=[\tilde w_j(y)- u_j(y)]+[u_j(y)-u_j(l)]$ to get
\be \label{eq:u-wtilde}
\begin{split}
\sup_{k \in \Z} \big\| u_{i+1}(k)-\tilde w_{i+1}(k)\big\|_{2m}^2 &\lesssim \frac{1}{n^{(\alpha-1)/\alpha}} \sum_{j=0}^i \sum_l \P_{i-j}^2(k, \, l) \sum_y \P(l, \, y) \cdot \big\|u_j(y) -u_j(l) \big\|_{2m}^2 \\
& \quad + \frac{1}{n^{(\alpha-1)/\alpha}} \sum_{j=0}^i \sum_l \P_{i-j}^2(k, \, l) \cdot \sup_{y \in \Z} \big\|u_j(y) -\tilde w_j(y) \big\|_{2m}^2.
\end{split}
\ee
The lemma follows by using Gronwall's inequality and Lemma \ref{lem:u:k-l}.
\end{proof}

Next consider the random field $w^*$ given by 
\be \label{eq:wstar}
 w^*_{i+1}(k) = \sum_l \P(k,l)\,  w^*_i(l) +\sum_l \P(k,l)\,  w^*_i(l) \cdot \left[\frac{e^{\beta \tilde\xi_{i+1}(k)}}{\bE e^{\beta \tilde\xi}}-1 \right]
\ee 
with initial profile $1$. Once again Theorem \ref{thm:u} holds for $w^*$ also. We compare $w^*$ with $\tilde w$ below.
\begin{lemma}
The following holds 
\bes
\sup_{i\le n,\, k \in \Z} \big\| \tilde w_i(k)- w^*_i(k)\big\|_{2m}^2 \lesssim \frac{1}{n^{(\alpha-1)/\alpha}}.
\ees
\end{lemma}
\begin{proof}
We write 
\bes\begin{split}
\tilde w_{i+1}(k) - w^*_{i+1}(k) & = \sum_{j=0}^i \sum_l\P_{i-j}(k, \, l)\left[\beta \tilde\xi_{j+1}(l)+1- \frac{e^{\beta \tilde\xi_{j+1}(l)}}{\bE e^{\beta \tilde\xi}}\right]\cdot \sum_y \P(l,\,y) \tilde w_j(y) \\
&\quad + \sum_{j=0}^i \sum_l\P_{i-j}(k, \, l)\left[\frac{e^{\beta \tilde\xi_{j+1}(l)}}{\bE e^{\beta \tilde\xi}}-1 \right]\cdot \sum_y \P(l,\,y) \left[w^*_j(y)- \tilde w_j(y)\right]
\end{split}\ees
One can check that for any $m\ge 1$ 
\be \label{eq:exp-lin}
\begin{split}
\left\|\frac{e^{\beta \tilde\xi}}{\bE e^{\beta \tilde\xi}}-1\right\|_{2m}^2 &\lesssim \frac{1}{n^{(\alpha-1)/\alpha}}, \quad \text{and}\\
\left\|\frac{e^{\beta \tilde\xi}}{\bE e^{\beta \tilde\xi}}-1-\beta\tilde\xi\right\|_{2m}^2 &\lesssim \frac{1}{n^{2(\alpha-1)/\alpha}}.
\end{split}
\ee
Using this we obtain for $i\le n$
\bes
\sup_{k \in \Z} \big\| \tilde w_{i+1}(k)- w^*_{i+1}(k)\big\|_{2m}^2 \lesssim \frac{1}{n^{(\alpha-1)/\alpha}} + \frac{1}{n^{(\alpha-1)/\alpha}}\sum_{j=0}^i \sum_l \P_{i-j}^2(k, \, l) \cdot \sup_{y \in \Z} \big\|\tilde w_j(y) -w^*_j(y) \big\|_{2m}^2.
\ees
The lemma follows by an application of Gronwall's inequality.
\end{proof} 
 
Finally we compare $w^*$ with $w$ in \eqref{eq:w:she}.
\begin{lemma} The following holds
\bes
\sup_{i\le n, \, k \in \Z} \big\| w^*_i(k)-w_i(k)\big\|_{2m}^2 \lesssim \frac{1}{n^{(\alpha-1)/\alpha}}.
\ees
\end{lemma}
\begin{proof}
We now have
\bes \begin{split}
w_{i+1}(k)-w^*_{i+1}(k) & = \sum_l \P_{i+1}(k,\,l) \cdot \left[\frac{e^{\beta \tilde\xi_0(l)}}{\bE e^{\beta \tilde\xi}}-1\right] \\
&\quad+ \sum_{j=0}^i \sum_l \P_{i-j}(k, \, l) \left[\frac{e^{\beta \tilde\xi_{j+1}(l)}}{\bE e^{\beta \tilde\xi}}-1\right] \cdot \sum_y \P(l,\,y) \big[w_j(l) -w^*_j(l) \big].
\end{split}\ees
Burkholder's inequality along with the first inequality in \eqref{eq:exp-lin} shows
\bes
\sup_{k \in \Z} \big\|  w_{i+1}(k)- w^*_{i+1}(k)\big\|_{2m}^2 \lesssim \frac{1}{n^{(\alpha-1)/\alpha}} + \frac{1}{n^{(\alpha-1)/\alpha}}\sum_{j=0}^i \sum_l \P_{i-j}^2(k, \, l) \cdot \sup_{y \in \Z} \big\| w_j(y) -w^*_j(y) \big\|_{2m}^2.
\ees
We can now use Gronwall's inequality to complete the proof.
\end{proof}
Our main result Theorem \ref{thm:main} shows that $u_1(0)$ of \eqref{eq:u} converges in distribution to $v_1(0)$ of \eqref{eq:she} 
with $\sigma(x)=\beta x$ and initial profile $1$. Furthermore the moments of $u$ converges to that of $v$. As a consequence of the previous lemmas and the discussion at the beginning of this section, the same holds for $M_n$. This completes the proof of
Theorem \ref{thm:part}. \qed

\section{Extensions} \label{sec:ext}
\subsection{Addition of a drift term}
We describe in this section an approximation to 
\be \label{eq:she:drift}
\partial_t v = -\nu(-\Delta)^{\alpha/2}\,v_t(x) + b(v_t(x))+ \sigma(v_t(x)) \cdot \dot{W}(t,x).
\ee
where $b:\R\to\R$ is also Lipschitz continuous with Lipschitz coefficient $\text{Lip}_b$. Recall that we scale
time by $n$, which suggests an approximation of the form
\be \label{eq:dis:sc:drift}
u^{(n)}_{i+1}(k)= \sum_{l\in \Z} \P(k,l) \, u_i^{(n)}(l) + b\big(u_i^{(n)}(k)\big)\cdot \frac{1}{n}+\sigma\big(u_i^{(n)}(k)\big)\cdot \frac{\xi_i(k)}{n^{(\alpha-1)/2\alpha}}.
\ee
We have the following
\begin{theorem} \label{thm:drift}
Let the conditions in Assumption \ref{cond1} hold, and fix an integer $m\ge 1$ so that $2m<2+\kappa$. Let $v_0$ be a continuous (random) function so that $\sup_x \mE |v_0(x)|^{2m}<\infty$ that is independent of $\xi$ and $\dot W$. Let $u^{(n)}$ be the solution to 
\eqref{eq:dis:sc:drift} with initial profile $u_0^{(n)}$. Then for each $t> 0, \, x\in \R$ we have $\bar u_t^{(n)}(x) \Rightarrow v_t(x)$, where $v$ is the solution to \eqref{eq:she:drift} with initial profile $v_0$. Furthermore we have 
\bes
\mE\big|\bar u^{(n)}_t(x) \big|^{2m} \to \mE\big| v_t(x)\big|^{2m} \quad \text{as } n \to \infty.
\ees
\end{theorem}
The proof of this theorem follows closely the proof of Theorem \ref{thm:main}. We provide the outline of the proof and leave the details to the reader. We start by checking that the conclusions of Theorem \ref{thm:u} continue to hold. Using the Cauchy-Schwarz inequality for the drift term we can check that the Picard iterates $w^{(p+1)}_{i+1}(k)$ in \eqref{eq:picard} satisfy
\bes
\begin{split}
\mE \left|w^{(p+1)}_{i+1}(k)-w^{(p)}_{i+1}(k)\right|^2 &\le \text{Lip}_b^2 \sum_{j=0}^i \sum_l \P_{i-j}(l-k)\cdot  \frac{\mE \big|w^{(p)}_{j}(l)-w^{(p-1)}_j(l)\big|^2}{n}\\
&\;\;+  \lip^2\sum_{j=0}^{i} \sum_{l \in \Z}\P_{i-j}^2(l-k)\cdot  \frac{\mE \big|w^{(p)}_{j}(l)-w^{(p-1)}_j(l)\big|^2}{n^{(\alpha-1)/\alpha}}.
\end{split}
\ees
With $\mathcal W(p)$ defined as in \eqref{w} we obtain now 
\[
\mathcal{W}^2(p+1) \le \text{Lip}_b^2 \cdot \mathcal{W}^2(p) \sum_{j=0}^{[nT]} \frac{e^{-\delta j/n}}{n} +  \lip^2\cdot \mathcal{W}^2(p)\sum_{j=0}^{[nT]}    \frac{e^{-\delta j/n}\cdot P \big(X_{j}=\tilde{X}_{j}\big)}{n^{(\alpha-1)/\alpha}}.
\]
We now choose $\delta$ large enough so that we obtain
\[
\mathcal W^{2}(p+1) \le \frac{1}{2} \mathcal W^{2}(p).\]
As before we can show the existence of a unique solution to \eqref{eq:dis:sc:drift} satisfying \eqref{eq:mom:unif}.
The second statement in Theorem \ref{thm:u} follows along similar lines. 

As in the proof of Theorem \ref{thm:main} we start with initial profile $v_0\equiv 0$, and impose conditions on $\gamma$ 
and $\theta$ as in \eqref{g:t:cond}. Our definition of $U$ in \eqref{eq:U} now includes an additional term
\be \label{eq:U:drift}\sum_{j=0}^{i-1} \sum_{l\in \Z}\P_{(i-1-j)[n^\theta]}\big((l-k)[n^\gamma]\big) \cdot b\big(u_{j[n^\theta]}(l[n^\gamma])\big)\cdot \frac{[n^\theta][n^\gamma]}{n}.\ee
We first consider Proposition \ref{prop:u:hol}. With the addition of the drift in \eqref{eq:dis:sc:drift} the difference between 
$\bar u_t(x)$ and $u_{r[n^\theta]}(z[n^\gamma])$ gives extra terms
\bes
\begin{split}
&\frac{1}{n} \sum_{j=r[n^\theta]}^{[nt]-1}\sum_{l\in \Z} \P_{[nt]-j-1}\big(l+[\mu nt] -[xn^{1/\alpha}]\big) \cdot b\big(u_j(l)\big)   \\
&\quad +\frac{1}{n} \sum_{j=0}^{r[n^\theta]-1} \sum_{l \in \Z} \Big[ \P_{[nt]-j-1}\big(l+[\mu nt] -[xn^{1/\alpha}]\big) - \P_{r[n^\theta]-j-1}\big(l- z[n^\gamma]\big)\Big]\cdot b\big(u_j(l)\big). 
\end{split}
\ees
The $\|\cdot\|_{2m}$ norm of the first term is of order $n^\theta/n$, while that of the second term is of order
\bes
\begin{split}
&\frac{1}{n} \sum_{j=0}^{r[n^\theta]-1} \sum_{l \in \Z} \Big\vert \P_{[nt]-r[n^\theta]+j}\big(l+[\mu nt] -[xn^{1/\alpha}]\big) - \P_{j}\big(l- z[n^\gamma]\big)\Big\vert \\
& \lesssim \frac{1}{n} \sum_{j=0}^{r[n^\theta]-1}  \sum_{w\in\Z} \P_{[nt]-r[n^\theta]}(w)  \cdot \sum_{l \in \Z}\Big\vert \P_{j}\big(l+[\mu nt] -[xn^{1/\alpha}]-w\big) - \P_{j}\big(l- z[n^\gamma]\big)\Big\vert.
\end{split}
\ees
Fix any $\epsilon>0$. We can remove the terms from $j=0$ to $j=[\epsilon n]$ with an error of order $\epsilon$. Also due to Lemma  \ref{lem:ld} we might restrict to $|w|\le c_3n^\theta$ up to a vanishing
error in $n$. Now consider the above expression over the range $j\ge [\epsilon n]$ and $|w|\le c_3n^\theta$. By Theorem \ref{llt} for each $l \in \Z$ and $j, w$ in the above ranges the difference goes to $0$ with $n$. Moreover we can bound the difference by the sum and get an expression which is uniformly bounded in $n$. Therefore by the dominated convergence theorem and the arbitrariness of $\epsilon$ we can conclude that the above expression goes to $0$ with $n$.

Before we move on to Proposition \ref{prop:u-U} we state a lemma we shall need.
\begin{lemma} \label{lem:periodic} Suppose additionally that $\gamma<a\theta$. Then for $i \ge 1$ we have
\[ \sum_{k\in \Z} \P_{i[n^\theta]} \big(k[n^\gamma]\big) \lesssim \frac{1}{n^\gamma}.\]
\end{lemma}
\begin{proof}
We shall use Theorem \ref{llt} for $|k|\le n^{-\gamma}(in^\theta)^{(1+a)/\alpha}$ and a large deviation estimate for 
$|k|>n^{-\gamma}(in^\theta)^{(1+a)/\alpha}$. One checks that 
\bes
\begin{split}
& \sum_{k} \P_{i[n^\theta]} \big(k[n^\gamma]\big) \\
& \lesssim \sum_{k} \Big[\frac{1}{(in^\theta)^{1/\alpha}}\cdot p_1\left(\frac{kn^\gamma}{(in^\theta)^{1/\alpha}}\right) + \frac{1}{(in^\theta)^{(1+a)/\alpha}}\Big] + P\left(\big|X_{i[n^\theta]}\big|\ge (in^\theta)^{(1+a)/\alpha}\right) \\
&\lesssim \frac{1}{n^\gamma}  + P\left(\big|X_{i[n^\theta]}\big|\ge (in^\theta)^{(1+a)/\alpha}\right).
\end{split}
\ees
For $1<\alpha<2$ one can use the result in \cite{heyd}
along with Lemma \ref{lem:doa} to conclude
\bes \begin{split} P\left(\big|X_{i[n^\theta]}\big|\ge (in^\theta)^{(1+a)/\alpha}\right) &\le (in^\theta) P\left(\big|X_{1}\big|\ge (in^\theta)^{(1+a)/\alpha}\right)   \lesssim \frac{1}{(in^\theta)^a}\lesssim \frac{1}{n^\gamma}.
\end{split}\ees
One can use Chebyshev's inequality to get a similar bound when $\alpha=2$.
\end{proof}
We now move on to Proposition \ref{prop:u-U}. The $\|\cdot \|_{2m}$ norm of the difference of
$u_{r[n^\theta]}(z[n^\gamma])$ and $U_{r[n^\theta]}(z[n^\gamma])$ arising from the drift terms in 
\eqref{eq:dis:sc:drift} and \eqref{eq:U:drift} is bound above by
\be \label{eq:u-U:drift}
\begin{split}
& \lesssim \frac{1}{n} \sum_{i,j} \sum_{k,l}\P_{r[n^\theta]-j-1}\big(l- z[n^\gamma]\big)\cdot \left\|u_j(l)-u_{i[n^\theta]}\left(k[n^\gamma]\right)\right\|_{2m}  \\
& \quad + \frac{1}{n} \sum_{i,j} \sum_{k,l}\Big\vert\P_{r[n^\theta]-j-1}\big(l- z[n^\gamma]\big)- \P_{(r-1-i)[n^\theta]}\big( (k-z)[n^\gamma]\big)\Big\vert
\end{split}
\ee
The first
term goes to $0$ by the H\"older continuity of $u$ argued above and earlier in Proposition \ref{prop:u:hol}. The second term also 
goes to $0$ by an argument similar to above. Due to Lemma \ref{lem:periodic}, the second expression in \eqref{eq:u-U:drift} is uniformly bounded in $n$. Therefore
we can apply the dominated convergence theorem to conclude that it goes to $0$.

To the random field $V$ in \eqref{eq:V} we add the drift term \eqref{eq:U:drift}. Therefore the difference in Proposition \ref{prop:V-U} arising from the drift terms is $0$. 

The random field $\bar V$ has an additional term 
\[\frac{[n^\theta] [n^\gamma]}{n^{1+\frac1\alpha}}\cdot\sum_{j=0}^{i-2} \sum_{l \in \Z} p_{\frac{(i-1-j)[n^\theta]}{n}}\left(\frac{l[n^\gamma]-(k+a_{i-1}-a_j)[n^\gamma]}{n^{1/\alpha}}\right)\cdot b\big(u_{j[n^\theta]}\left(l [n^\gamma]\right)\big).   \]
The bound on the $\|\cdot\|_{2m}$ norm arising from the difference of this from the drift term in $V$ is
\be \label{eq:V-Vbar:drift} \frac{n^{\theta+\gamma}}{n^{1+\frac{1}{\alpha}}}\sum_{j=1}^{r-1} \sum_{l \in \Z} \left| n^{\frac{1}{\alpha}}\P_{j[n^\theta]} \left(l[n^\gamma]\right) -p_{\frac{j[n^\theta]}{n}}\left(\frac{(l-a_{r-1}+a_{r-1-j})[n^\gamma]}{n^{1/\alpha}}\right) \right|\ee
We claim that the expression is uniformly
bounded. Indeed the sum of the terms involving the density $p$ is uniformly bounded by a Riemann sum argument.
As for the sum of the terms involving the transition kernel $\P$, it is also uniformly bounded by Lemma \ref{lem:periodic}
if we assume $\gamma<\alpha\theta$.
The convergence of \eqref{eq:V-Vbar:drift} to $0$ thus follows from the dominated convergence theorem, similar to what was argued earlier in this outline.

Finally we consider 
\[v_t(x) = \int_0^t \int_{\R} p_{t-s}(y-x) \sigma\big(v_s(y)\big) \mathcal W(\d s\,\d y) + \int_0^t \int_{\R} p_{t-s}(y-x) b\big(v_s(y)\big) \d s\,\d y. \]
One can again show that $v$ is H\"older continuous in both space and time.  We can similarly define the random fields $\tilde v$ and $\bar v$ by including 
a drift term to each of them. The drift terms we add to the definition of $\tilde v$ and $\bar v$ in equations \eqref{eq:tildeV} and 
\eqref{eq:barv} are 
\bes
\begin{split}
& \frac{[n^\gamma]}{n^{1/\alpha}} \int_0^{t-\frac{[n^\theta]}{n}} \sum_{l \in \Z} p_{t-s} \left(\frac{(l-k)[n^\gamma]}{n^{1/\alpha}}\right)\cdot b\left(v_s\left(\frac{l[n^\gamma]}{n^{1/\alpha}}\right)\right)\, \d s,\quad \text{and} \\
& \frac{[n^\theta][n^\gamma]}{n^{1+\frac1\alpha}}\sum_{j=0}^{i-2} \sum_{l \in \Z} p_{\frac{(i-1-j)[n^\theta]}{n}}\left(\frac{(l-k)[n^\gamma]}{n^{1/\alpha}}\right)\cdot b\left(v_{\frac{j[n^\theta]}{n}}\left(\frac{l [n^\gamma]}{n^{1/\alpha}}\right)\right)
\end{split}
\ees
respectively. One can show again that both
\[\left\|v_t(x) - \tilde v_t\left(\frac{[n^\gamma]}{n^{1/\alpha}}\left[\frac{xn^{1/\alpha}}{[n^\gamma]}\right]\right)\right\|_{2m} \quad \text{and} \quad\left\|\tilde{v}_t\left(\frac{k[n^\gamma]}{n^{1/\alpha}}\right)-\bar{v}_{\frac{r[n^\theta]}{n}}\left(\frac{k[n^\gamma]}{n^{1/\alpha}}\right)  \right\|_{2m}^2 \]
go to $0$ with $n$. Collecting all our bounds we can then use \eqref{eq:final} and apply Gronwall's inequality to prove 
Theorem \ref{thm:drift} in the case $v_0\equiv 0$. The case of general $v_0$ follows as earlier.

\subsection{Dirac initial condition} \label{sec:dirac}
Consider \eqref{eq:dis:sc} with initial profile 
\be \label{dirac:initial} u_0^{(n)}(k)=n^{1/\alpha} \cdot \ind\{k=0\}.  \ee
Our main result of this section is the following 
\begin{theorem} \label{thm:dirac}
Let the conditions in Assumption \ref{cond1} hold, and suppose $|\sigma(x)|\le \lip |x|$ for all $x$. Fix an integer $m\ge 1$ so that $2m<2+\kappa$. Let $u^{(n)}$ be the solution to 
\eqref{eq:dis:sc} with initial profile \eqref{dirac:initial}. Then for each $t> 0, \, x\in \R$ we have $\bar u_t^{(n)}(x) \Rightarrow v_t(x)$, where $v$ is the solution to \eqref{eq:she} with initial profile $\delta_0$. Furthermore we have 
\bes 
\mE\big|\bar u^{(n)}_t(x) \big|^{2m} \to \mE\big| v_t(x)\big|^{2m} \quad \text{as } n \to \infty.
\ees
\end{theorem}
To avoid unnecessarily complicated notation we assume from now on that $\mu=0$. The reader can convince
theirselves that the proof of the above theorem continues to hold for nonzero $\mu$. 

Due to the singularity of the density $p$ at time $0$ we can obtain uniform moment bounds
on $u$ only if we consider a compact time interval away from $0$. More precisely we have the following
\begin{proposition} \label{pr:umbd} Let $m\in \N$ so that $2m< 2+\kappa$. Fix $\epsilon>0$. We have
\[ \sup_{\epsilon n\le i\le n,\, k \in \Z}\mE | u_{i}(k)|^{2m} <\infty,\]
uniformly in $n$.
\end{proposition}
\begin{proof}
Burkholder's inequality gives for $i\ge 1$
\[
\left\| u_{i}(k)\right\|_{2m}^2 \le  c_0 n^{2/\alpha} \P_{i}^2(k,\,0) + c_0\sum_{j_1=0}^{i-1} \sum_{l_1} \frac{n^{2/\alpha}\P_{i-j_1}^2(k, \, l_1)}{n^{(1+\alpha)/\alpha}} \cdot \left\|u_{j_1}(l_1)\right\|_{2m}^2.\]
If we apply the above recursively we obtain
\be \label{eq:u:recur}
\begin{split}
\left\| u_{i}(k)\right\|_{2m}^2
&\le c_0 n^{2/\alpha} \P_{i}^2(k,\,0) + c_0^2\sum_{j_1=0}^{i-1}  \sum_{l_1}  \frac{n^{2/\alpha}\P_{i-j_1}^2(k, \, l_1)}{n^{(1+\alpha)/\alpha}} \cdot n^{2/\alpha}\P_{j_1}^2(l_1,0) \\
&\quad+ c_0^3\sum_{0\le j_2<j_1<i}\,  \sum_{l_1,\, l_2}  \frac{n^{2/\alpha}\P_{i-j_1}^2(k, \, l_1)}{n^{(1+\alpha)/\alpha}} \cdot \frac{n^{2/\alpha}\P_{j_1-j_2}^2(l_1,l_2)}{n^{(1+\alpha)/\alpha}} \cdot n^{2/\alpha} \P_{j_2}^2(l_2,0)+\cdots 
\end{split}\ee
Suppose $i=[c_1 n],\,  k=[c_2n^{1/\alpha}]$ for some $c_1>0$ and $c_2 \in \R$. One can follow the argument in Section 7 of \cite{cara-sun-zygo} to deduce that each of the terms converge to the corresponding iterated integrals involving the density $p$. For example the second term on the right would converge to 
\be \label{eq:it:int} c_0^2 \int_0^{c_1}\d s\,\int_{\R}\d y\,  p_{c_1-s}^2(c_2-y) p_s^2(y). \ee
Furthermore the tail sums are uniformly bounded in $n$. 

Now when we consider the moments of the  parabolic Anderson model $\partial_t v = -\nu(-\Delta)^{\alpha/2} v + c_0 v \dot{W}$
we get an infinite series involving the integrals of $p^2$.  Therefore we can bound the moments of $u_i(k)$ from above by the moments $v_{c_1}(c_2)$.  Equation (4.15) in \cite{chen-dala-2} states that the moments of $v$ are uniformly bounded in a compact subset of $\R_+\times \R$. We next show that the iterated integrals of $p$ above are maximized when $c_2=0$. Indeed let us consider \eqref{eq:it:int}. By Plancheral's theorem this is equal to a constant mulitple of 
\[ \int_0^{c_1} \d s \int_{\R} \d \xi \,  \big\vert\mathcal F\big[p_{c_1-s}(c_2-\cdot) p_s(\cdot)\big](\xi)\big\vert^2,\]
where $\mathcal F$ denotes the Fourier transform in the spatial variable. Using $\mathcal F(fg)=\mathcal F(f)*\mathcal F(g)$
for $L^2$ functions $f$ and $g$ one obtains
\bes\begin{split} &\int_0^{c_1} \d s \int_{\R} \d \xi \,  \big\vert\big\lbrace\mathcal F\big[p_{c_1-s}(c_2-\cdot)\big]*\mathcal F\big[p_s(\cdot)\big]\big\rbrace(\xi)\big\vert^2 \\
&=\int_0^{c_1} \d s \int_{\R} \d \xi \, \big\vert\big\lbrace e^{-\mathrm i c_2 \cdot}e^{-\nu(c_1-s)|\cdot|^\alpha}* e^{-\nu s|\cdot|^\alpha}\big\rbrace (\xi) \big\vert^2 \\
&\lesssim \int_0^{c_1} \d s \int_{\R} \d \xi \, \big\vert \big\lbrace e^{-\nu(c_1-s)|\cdot|^\alpha}* e^{-\nu s|\cdot|^\alpha}\big\rbrace (\xi)\big\vert^2  \\
&= \int_0^{c_1} \d s \int_{\R} \d \xi \,  \big\vert\mathcal F\big[p_{c_1-s}(\cdot) p_s(\cdot)\big](\xi)\big\vert^2.\end{split}
\ees
A similar argument holds for the higher iterated integrals. Therefore the moments of $v$ are in fact uniformly bounded in a compact subset of time. This proves the proposition.
\end{proof}
\begin{proof}[Proof of Theorem \ref{thm:dirac}] We aim to use the arguments in the proof of our main 
theorem \ref{thm:main} here also. For this we remove some terms from the noise term since the moments of $u, v$ are ``large" near
time $0$. Let us make this more precise. Fix a small $\epsilon>0$. We claim that the random field
\[ v_t(x) = p_t(x) + \int_0^t \int_{\R} p_{t-s}(x-y) \sigma\big(v_s(y)\big) W(\d s\, \d y)\]
is close to 
\[ \tilde v_t(x) = p_t(x) + \int_\epsilon^t \int_{\R} p_{t-s}(x-y) \sigma\big(v_s(y)\big) W(\d s\, \d y) \]
in the sense that 
$\big\| v_t(x) - \tilde v_t(x)\big\|_{2m}^2\to 0$ as $\epsilon \to 0$.  Note that
\[  \big\|v_t(x)\big\|_{2m}^2 \le C p_t^2(x) + C \int_0^t \int_{\R} p^2_{t-s}(x-y) \big\|v_s(y)\big\|_{2m}^2 \d s\, \d y, \]
which when iterated gives an infinite series which can be shown to be finite, see \cite{chen-dala-2}. In particular 
\[ \int_0^t \int_{\R} p^2_{t-s}(x-y) \big\|v_s(y)\big\|_{2m}^2 \d s\, \d y <\infty,\]
from which our claim follows. Similarly
\[ u_{[nt]}\big([xn^{1/\alpha}]\big) = n^{1/\alpha}\P_{[nt]}\big([xn^{1/\alpha}]\big) + \sum_{j=0}^{[nt]-1} \sum_l \P_{[nt]-j-1}\big([xn^{1/\alpha}],\,l\big) \cdot \sigma\big(u_j(l)\big) \cdot \frac{\xi_j(l)}{n^{(\alpha-1)/2\alpha}} \]
is close to
\[\tilde u_{[nt]}\big([xn^{1/\alpha}]\big) = n^{1/\alpha}\P_{[nt]}\big([xn^{1/\alpha}]\big) + \sum_{j=[\epsilon n]}^{[nt]-1} \sum_l \P_{[nt]-j-1}\big([xn^{1/\alpha}],\,l\big) \cdot \sigma\big(u_j(l)\big) \cdot \frac{\xi_j(l)}{n^{(\alpha-1)/2\alpha}}\]
in the sense $\sup_n\big\| u_{[nt]}([xn^{-1/\alpha}]) - \tilde u_{[nt]}([xn^{-1/\alpha}])\big\|_{2m}^2\to 0$ as $\epsilon \to 0$. This can be seen by following the argument in Proposition \ref{pr:umbd}.

Note first that $n^{1/\alpha}\P_{[nt]}([xn^{1/\alpha}]) \rightarrow p_t(x)$ as $n \to \infty$. Away from time $t=0$ the moments of $u$ and $v$ are uniformly bounded, that is for each $\epsilon>0$ and uniformly in $n$
\bes \sup_{i\ge\epsilon n, \, k\in \Z}\big\|u_i(k)\big\|_{2m}^2  < \infty,\qquad
 \sup_{i\ge\epsilon n, \, k\in \Z}\big\|v_{\frac{i}{n}}(kn^{-1/\alpha})\big\|_{2m}^2  <\infty. \ees
Thus the arguments in the proof of Theorem \ref{thm:main} can be used and we can show that 
\[ \left\|\tilde u_{[nt]}([xn^{-1/\alpha}]) - \tilde v_t(x)\right\|_{2m}^2 \to 0 \mbox{ as } n\to \infty.\]
The theorem follows from this.
\end{proof}

\subsection{Proof of parts 2 and 3 in Theorem \ref{thm:part}}
Consider now the random field 
\be \label{eq:w:delta}
w_i(k)= \frac{n^{1/\alpha} E_k \left[\mathbf 1\{X_i=0\}\cdot \exp \left( \sum_{j=0}^i \beta \cdot \tilde\xi_{i-j}(X_j)\right)\right]}{\left[\bE e^{ \beta\tilde\xi}\right]^{i+1}},\quad i \ge 0,\, k \in \Z.
\ee
This satisfies \eqref{eq:w:she} but now
with initial profile 
\[ w_0(k) = n^{1/\alpha} \mathbf 1\{k=0\}\cdot \frac{e^{\beta \tilde\xi_0(k) }}{\mathbb E e^{\beta \tilde\xi}}. \]
Following the argument in Section \ref{sec:part:1} we show that $w$ is close to the 
random field $u$ in \eqref{eq:u} with initial profile 
\bes
u_0(k)=n^{1/\alpha}\mathbf \{k=0\}.
\ees
We start with a bound on the H\"older continuity of $u$. One checks
\bes\begin{split}
&\big\|u_i(k)-u_i(l)\big\|_{2m}^2 \\
&\lesssim n^{2/\alpha}\big[\P_i(k,\,0)-\P_i(l,\,0)\big]^2 +\sum_{j=0}^{i-1}\sum_y \frac{n^{2/\alpha}\big[\P_{i-1-j}(k,\,y)- \P_{i-1-j}(l,\,y)\big]^2}{n^{(\alpha+1)/\alpha}} \cdot \big\| u_j(y)\big\|_{2m}^2.
\end{split}
\ees
If we iterate this we obtain 
\be\begin{split} \label{eq:u:k-l:2}
&\big\|u_i(k)-u_i(l)\big\|_{2m}^2 \\
&\lesssim n^{2/\alpha}\big[\P_i(k,\,0)-\P_i(l,\,0)\big]^2 + \sum_{j_1=0}^{i-1}\sum_{y_1} \frac{n^{2/\alpha}\big[\P_{i-1-j_1}(k,\,y_1)- \P_{i-1-j_1}(l,\,y_1)\big]^2}{n^{(\alpha+1)/\alpha}} \cdot n^{2/\alpha} \P_{j_1}^2(y_1, 0)\\
& \; +  \sum_{0\le j_2<j_1<i}\,\sum_{y_1, y_2} \frac{n^{2/\alpha}\big[\P_{i-1-j_1}(k,\,y_1)- \P_{i-1-j_1}(l,\,y_1)\big]^2}{n^{(\alpha+1)/\alpha}} \cdot \frac{  n^{2/\alpha} \P_{j_1-j_2}^2(y_1, y_2)}{n^{(\alpha+1)/\alpha}}\cdot n^{2/\alpha} \P_{j_2}^2(y_2, 0)+\cdots
\end{split}
\ee
From this one can show
\begin{lemma} \label{lem:hol:1step} For each $\epsilon>0$ 
\bes \sup_{\epsilon n\le i \le n,\, k \in \Z} \sum_l \P(k, l) \cdot \big\| u_i(k) - u_i(l) \big\|_{2m}^2 =o(1)\quad \text{as } n\to \infty \ees
\end{lemma}
\begin{proof} It is enough to show that $\big\| u_i(k) - u_i(l) \big\|_{2m}^2 =o(1)$ for fixed $k$ and $l$ 
when $\epsilon n\le i \le n$. The arguments in Section 7 of \cite{cara-sun-zygo} show that the non vanishing 
contribution to \eqref{eq:u:k-l:2} come from terms where the successive differences $i-j_1$, $j_1-j_2, \cdots$
are of order $n$. Due to Theorem \ref{llt} $n^{2/\alpha}\big[\P_{i}(k, y) -\P_i(l, y)\big]^2$ is of order
$o(1)$ if $i$ is of order $n$. This along with the uniform moment bound of $u$ obtained from \eqref{eq:u:recur} 
is enough to prove the lemma.
\end{proof}

As before we compare first $u$ with $\tilde w$ where $\tilde w$ now satisfies
\eqref{eq:wtilde} but with initial profile $\tilde w_0(k) = n^{1/\alpha} \mathbf 1\{ k=0\}$.
We have
\begin{lemma} We have
\bes
\sup_{k \in \Z} \big\| u_n(k)-\tilde w_n(k) \big\|_{2m}^2 =o(1)\quad \text{as }n\to \infty.
\ees
\end{lemma}
\begin{proof} We iterate \eqref{eq:u-wtilde}, so as to obtain a bound which does not involve the difference between $u$ and $\tilde w$. Thus we obtain 
\be \label{76}
\begin{split}
& \sup_{k \in \Z} \big\| u_n(k)-\tilde w_n(k) \big\|_{2m}^2  \\
& \lesssim \sum_{j_1=0}^{n-1} \sum_l \frac{n^{2/\alpha}\P_{n-1-j_1}^2(k,\,l_1)}{n^{(\alpha+1)/\alpha}} \sum_y \P(l_1,\,y) \cdot \big\| u_{j_1}(y) -u_{j_1}(l_1) \big\|_{2m}^2 \\
&+  \sum_{0\le j_2<j_1\le n-1}\, \sum_{l_1, l_2} \frac{n^{2/\alpha}\P_{n-1-j_1}^2(k,\,l_1)}{n^{(\alpha+1)/\alpha}}\cdot \frac{n^{2/\alpha} \P^2_{j_1-j_2}(l_1, \,l_2)}{n^{(\alpha+1)/\alpha}}\sum_y \P(l_2,\,y) \cdot \big\| u_{j_2}(y) -u_{j_2}(l_2) \big\|_{2m}^2 +\cdots
\end{split}
\ee
Consider the $q$'th sum in above. We have the trivial bound
\[ \big\| u_{j_q}(y) -u_{j_q}(l_q)\big\|_{2m}^2 \lesssim \big\| u_{j_q}(y)\big\|_{2m}^2 + \big\| u_{j_q}(l_q)\big\|_{2m}^2. \]
If we replace this bound in each of the terms in \eqref{76}, and then use the bound \eqref{eq:u:recur} we get an expression which is similar to the expression on the right hand side of \eqref{eq:u:recur} minus the first term. Using arguments similar to those 
in \cite{cara-sun-zygo} one can conclude that this expression is uniformly bounded. Moreover each of the terms in \eqref{76}
go to $0$, thanks to Lemma \ref{lem:hol:1step} and the arbitrariness of $\epsilon$.  Thus, by the dominated convergence theorem
the lemma is proved.
\end{proof}
Next we compare $\tilde w$ with $w^*$ in \eqref{eq:wstar} with initial profile
$w^*_0(k)=n^{1/\alpha}\mathbf 1\{k=0\}$. We have
\begin{lemma} We have
\bes
\sup_{k \in \Z} \big\| w_n^*(k)-\tilde w_n(k) \big\|_{2m}^2 =o(1) \quad \text{as }n\to \infty .
\ees
\end{lemma}
\begin{proof} We now have 
\bes\begin{split}
\sup_{k \in \Z} \big\| \tilde w_{i+1}(k)- w^*_{i+1}(k)\big\|_{2m}^2 &\lesssim  \frac{1}{n^{(\alpha-1)/\alpha}} \sum_{j=0}^i  \sum_l \frac{n^{2/\alpha}\P_{i-j}^2(k, l)}{n^{(\alpha+1)/\alpha}}\cdot \sum_y \P(l, y) \big\| w_j^*(y) \big\|_{2m}^2\\
&\quad + \frac{1}{n^{(\alpha-1)/\alpha}}\sum_{j=0}^i \sum_l \P_{i-j}^2(k, \, l) \cdot \sup_{y \in \Z} \big\|\tilde w_j(y) -w^*_j(y) \big\|_{2m}^2.
\end{split}
\ees
Again one can check that the moment bounds of $w^*$ satisfy a relation similar to \eqref{eq:u:recur}. Iterating the above proves the lemma.
\end{proof}
Finally we compare $w^*$ with $w$
\begin{lemma} We have
\bes\sup_{k \in \Z} \big \| w_n^*(k) -w_n(k) \big\|_{2m}^2 =o(1) \text{as } n\to \infty.\ees
\end{lemma}
\begin{proof} We now have
\bes
\sup_{k \in \Z} \big\|  w_{i+1}(k)- w^*_{i+1}(k)\big\|_{2m}^2 \lesssim \frac{n^{2/\alpha} \P_{i+1}^2(k, 0)}{n^{(\alpha-1)/\alpha}} + \frac{1}{n^{(\alpha-1)/\alpha}}\sum_{j=0}^i \sum_l \P_{i-j}^2(k, \, l) \cdot \sup_{y \in \Z} \big\| w_j(y) -w^*_j(y) \big\|_{2m}^2.
\ees
Iterating the above proves the result.
\end{proof}
Thanks to Theorem \ref{thm:dirac} we know that $u_n([xn^{1/\alpha}])$ converges in distribution to $g_1(x)$, where $\partial_t g=-\nu(-\Delta)^{\alpha/2}g + \beta g \dot{W}$ and $g_0 =\delta_0$. The above lemmas then imply that the same holds for $w_n([xn^{1/\alpha}])$ and consequently for $n^{1/\alpha} M_n^{(\tilde \xi, [-xn^{1/\alpha}])}$. The second part 
of Theorem \ref{thm:part} follows by a spatial translation of the noise.

Denote $u^{(1)}$ be the solution to \eqref{eq:u} with initial profile $1$ and let $u^{(2)}$ be the solution to \eqref{eq:u} with 
initial profile $\delta_0$. Using the same probability space constructed in Theorem \ref{thm:couple} for both $u^{(1)}$
and $u^{(2)}$ we see that $\big(u^{(1)}_n([xn^{1/\alpha}]), u^{(2)}_n([xn^{1/\alpha}]) \big)$ converges jointly in 
distribution to $(v_1(x), g_1(x))$. Because of the previous lemmas the same holds for $\big(
w^{(1)}_n([xn^{1/\alpha}]), w^{(2)}_n([xn^{1/\alpha}]) \big)$, where $w^{(1)}$ solves \eqref{eq:w} and 
$w^{(2)}$ solves \eqref{eq:w:delta}. By time reversal and independence of the random variables $\xi$ we can conclude
that $\big(M_n^{(\tilde \xi)},\, n^{1/\alpha}M_n^{(\tilde \xi, \, [xn^{1/\alpha}]}\big)$ converges jointly to $(v_1(-x), g_1(-x))$. The third statement
of Theorem \ref{thm:part} follows from this and a spatial translation of the noise since $n^{1/\alpha}\P_n^{\tilde\xi, \beta}([xn^{1/\alpha}])$ is simply the ratio of $ n^{1/\alpha}M_n^{(\tilde \xi, \, [xn^{1/\alpha}]}$ and $M_n^{(\tilde \xi)}$. \qed

\subsection{Some more extensions}
There are a few other possible extensions of the above results.
\begin{enumerate}
\item The conditions on $\xi$ can be relaxed. For example it is not necessary for them to be i.i.d.. It is enough that they are independent and satisfy conditions in \eqref{cond:xi} for the weak convergence in Theorem \ref{thm:main} to hold, see for example \cite{boro}. However we need to assume stronger conditions to obtain convergence of high moments. 
\item We could relax the conditions on $\eta$, for example we could take them to be independent with {\it slowly varying} variances (as in equation 2.10 in \cite{bala-rass-sepp}). The arguments in Section \ref{sec:char} show that the contributions of the noise $\xi$ and initial configuration $\eta$ can be treated separately, so one could even consider correlated $\eta$ as in \cite{sepp-zhai}.
\item With some work, it should be possible to extend Theorem \ref{thm:dirac} to more general compactly supported measures.
\item We can consider weak approximations of the stochastic heat equation with spatially colored noise; depending on the noise these can exist in higher dimensions (see \cite{foon-jose-li}). The key ingredient in the proof of our main theorem \ref{thm:main} was the coupling 
theorem \ref{thm:couple}. A starting point would be Section 3.1 of \cite{conu-jose-khos-shiu} where a large family of colored noises are constructed by convoluting appropriate functions with white noise. These include {\it Riesz noises} (see \cite{foon-jose-li}) which interpolate between spatially smooth noise and white noise. The construction also suggests that one should then consider random variables of the form $\sum_x a_x\cdot\xi_i(k+x)$ for appropriate $a_x, x\in \Z^d$ as our discrete noise; here $\xi_i(k)$ are i.i.d. random variables satisfying \eqref{cond:xi} as before. We could also
consider temporally correlated noise. One could then apply the above to study 
intermediate disorder regimes for directed polymers in a correlated random environment. 

\end{enumerate}

\appendix

\section{Appendix: A local limit theorem} \label{sec:app:A}
The following local limit theorem is a modification of Proposition 3.3 in \cite{foon-jose-li}. Below $p_t$ is the density 
for the Stable($\alpha$) process with generator $-\nu(-\Delta)^{\alpha/2}$.
\begin{theorem}\label{llt} Suppose Assumption \ref{cond1} holds. Then for any $0\le b\le 1$ and any $c>0$
\be \label{eq:llt} \sup_{k \in \Z} \;\sup_{|x-(k-\mu[nt])|\le cn^{(1-b)/\alpha}}\left | n^{\frac{1}{\alpha}} \P_{[nt]}(k) - p_{\frac{[nt]}{n}}\left(x n^{-\frac{1}{\alpha}}\right)\right| \lesssim \frac{1}{n^{a/\alpha} t^{(1+a)/\alpha}} + \frac{1}{n^{b/\alpha}t^{2/\alpha}},\ee
uniformly for $1/n\le t\le T$.
\end{theorem}
\begin{proof} To simplify notation we denote $\tilde t=[nt]/n$. We first bound the expression on the left for $x=k-\mu[nt]$. Fourier inversion gives
\bes\begin{split}
p_{\tilde t}(x) = \frac{1}{2\pi} \int_{\R} e^{-\mathrm{i} xz} e^{-\nu \tilde t|z|^\alpha} \d z,
\end{split}\ees
and 
\bes
\P_{[nt]}(k) = \frac{1}{2\pi} \int_{-\pi}^{\pi} e^{-\mathrm{i} k z} \cdot \left[\phi(z)\right]^{[nt]}.
\ees
Therefore
\bes\begin{split}
&(2\pi)\left | n^{\frac{1}{\alpha}} \P_{[nt]}(k) - p_{\tilde t}\left((k-\mu[nt])n^{-\frac{1}{\alpha}}\right)\right| \\
&\qquad \le
\int_{ \left[-\pi n^{1/\alpha},\pi n^{1/\alpha}\right]^c}e^{-\nu \tilde t|z|^\alpha} \d z + \int_{-\pi n^{1/\alpha}}^{\pi n^{1/\alpha}} \left|e^{-\nu \tilde t|z|^\alpha}- \tilde\phi\left(\frac{z}{n^{1/\alpha}}\right)^{[nt]}\right| \d z. 
\end{split}\ees
The first term can be bound just as term $I_2$ in Proposition 3.3 of \cite{foon-jose-li}, and we get
\bes
\int_{ \left[-\pi n^{1/\alpha},\pi n^{1/\alpha}\right]^c}e^{-\nu \tilde t|z|^\alpha} \d z \lesssim \frac{1}{n^{a/\alpha} t^{(1+a)/\alpha}}.
\ees
For the second term we split the region of integration depending on whether or not $z$ is in
\bes
A_{t,n} :=\Big\{z\in \R; \, |z|\le \frac{n^{a/\{\alpha(a+\alpha)\}}}{t^{1/(a+\alpha)}} \Big\}.
 \ees
Our assumptions on the characteristic function $\phi$ imply
\bes
\Big\vert\tilde\phi\big(\frac{z}{n^{1/\alpha}}\big)\Big\vert  \le 1- c_1 \frac{|z|^\alpha}{n}
\ees
on $|z|\le \pi n^{1/\alpha}$, and so it follows as in \cite{foon-jose-li} that
\bes
\int_{A_{t,n}^c\cap [-\pi n^{1/\alpha},-\pi n^{1/\alpha}]} \left|e^{-\nu \tilde t|z|^\alpha}- \tilde\phi\left(\frac{z}{n^{1/\alpha}}\right)^{[nt]}\right| \, \d z  \lesssim \frac{1}{n^{a/\alpha} t^{(1+a)/\alpha}}.
\ees
We next need to consider the above integrand over the region $z \in A_{t,n}$. First observe
\bes\begin{split}
e^{-\nu \tilde t|z|^\alpha}- \tilde\phi\left(\frac{z}{n^{1/\alpha}}\right)^{[nt]} &= e^{-\nu t|z|^\alpha} - \exp\left[[nt] \log\tilde\phi\left(\frac{z}{n^{1/\alpha}}\right) \right] \\
&=e^{-\nu t|z|^\alpha} \left\lbrace1-\exp\left[[nt] \log\tilde\phi\left(\frac{z}{n^{1/\alpha}}\right) +\nu \tilde t|z|^\alpha\right] \right\rbrace \\
&= e^{-\nu t|z|^\alpha} \left\lbrace1-\exp\left[[nt] \log\left(1-\nu\frac{|z|^\alpha}{n}+\mathcal{D}\left(\frac{z}{n^{1/\alpha}}\right)\right) +\nu \tilde t|z|^\alpha\right] \right\rbrace. 
\end{split}\ees
It is easy to see that $[nt]\mathcal{D}\big(z/n^{1/\alpha}\big)$  is bounded on $A_{t,n}$,  and so 
\bes\begin{split}
\int_{A_{t,n}\cap [-\pi n^{1/\alpha},-\pi n^{1/\alpha}]} \left|e^{-\nu \tilde t|z|^\alpha}- \tilde\phi\left(\frac{z}{n^{1/\alpha}}\right)^{[nt]}\right| \, \d z 
&\lesssim \int_{\R} e^{-\nu \tilde t|z|^\alpha} \cdot (nt) \left| \frac{z}{n^{1/\alpha}}\right|^{a+\alpha} \\
&\lesssim \frac{1}{n^{a/\alpha} t^{(1+a)/\alpha}}.
\end{split}\ees
We thus have the required bound in the case $x=k-\mu[nt]$. In the case of a general $x$ such that $|x-(k-\mu[nt])|\le n^{(1-b)/\alpha}$ we have
\bes\begin{split}
&\left|p_{\tilde t}(xn^{-\frac1\alpha})-p_{\tilde t}((k-\mu[nt])n^{-\frac1\alpha})\right| \\
 & \le \int_{\R} \left | e^{-\mathrm{i}z(k-\mu[nt])n^{-\frac{1}{\alpha}}}- e^{-\mathrm{i}zxn^{-\frac1\alpha} } \right| e^{-\nu \tilde t|z|^\alpha}\, \d z \\
&\lesssim \int_{\R} 1\wedge \frac{|z|}{n^{b/\alpha}} \cdot e^{-\nu \tilde t|z|^\alpha} \d z \\
&\le \frac{1}{n^{b/\alpha}t^{2/\alpha}}\int_{|w|\le n^{b/\alpha}t^{1/\alpha}} |w| e^{-\nu |w|^\alpha} \d w  +\frac{1}{t^{1/\alpha}} \int_{|w|> n^{b/\alpha}t^{1/\alpha}} e^{-\nu |w|^\alpha}\d w\\
&\lesssim \frac{1}{n^{b/\alpha}t^{2/\alpha}}.
\end{split}\ees
This completes the proof. 
\end{proof}
The local limit theorem has the following consequence.
\begin{corollary} \label{cor:green} Let Assumption \ref{cond1} hold. Suppose $a_n$ is an integer valued sequence such that 
\bes
\frac{a_n-\mu[nt]}{n^{1/\alpha}} \rightarrow a,
\ees
as $n \to \infty$ for some constant $a$. Then for fixed $t>0$
\bes
\frac{1}{n^{(\alpha-1)/\alpha}} \sum_{i=0}^{[nt]} \P_i(a_n) \rightarrow \int_0^t\frac{\d s}{s^{1/\alpha}}\, p_1\left(\frac{a}{s^{1/\alpha}}\right) \; \text{as } n \to \infty.
\ees
\end{corollary}
\begin{proof} We write 
\bes\begin{split}
\frac{1}{n^{(\alpha-1)/\alpha}} \sum_{i=1}^{[nt]} \P_i(a_n) &= \frac{1}{n}\sum_{i=1}^{[nt]}   n^{\frac{1}{\alpha}}\P_i(a_n) \\
&= \frac{1}{n} \sum_{i=1}^{[nt]} \left[p_{\frac{i}{n}}\left(\frac{a_n-\mu[nt]}{n^{1/\alpha}}\right)+ \frac{n^{1/\alpha}}{i^{(1+a)/\alpha}}\right]
\end{split}\ees
Use the scaling property: for any $c>0$
\be\label{eq:scale}
p_t(x) = c p_{c^\alpha t}(cx)
\ee
to write the above as
\bes
\frac{1}{n} \sum_{i=1}^{[nt]} \frac{1}{(i/n)^{1/\alpha}}\cdot p_1\left(\left(\frac{n}{i}\right)^{1/\alpha} \cdot\frac{a_n-\mu [nt]}{n^{1/\alpha}}\right) +o(1).
\ees
The rest is a Riemann sum approximation.
\end{proof}

\section{Appendix: Bounds required for Proposition \ref{prop:u-U} } \label{sec:app:B}
The main results of this section are Lemmas \ref{lem:2_1} and \ref{lem:2_2} which were needed in the proof
of Proposition \ref{prop:u-U}.
\begin{lemma} \label{lem:green} The following holds 
\bes
\sup_{n}\sup_{w\in \Z} \frac{[n^\theta]}{n^{(\alpha-1)/\alpha}}\sum_{i=1}^{r-1} P\big(Y_{i[n^\theta]}=w\big) \le \sup_{n} \frac{[n^\theta]}{n^{(\alpha-1)/\alpha}}\sum_{i=1}^{r-1} P\big(Y_{i[n^\theta]}=0\big) <\infty.
\ees
\end{lemma}
\begin{proof}The first inequality is a simple consequence of the fact that
\[P(Y_j=w) = \frac{1}{2\pi}\int_{-\pi}^{\pi} e^{-\mathrm i w z} \vert \phi(z)\vert^{2j} \d w \le P(Y_j=0)\]
for any $w$, since the characteristic function of $Y_j$ is $\vert \phi(z)\vert^{2j}$. As for the finiteness of the sum
observe that $P(Y_j=0)$ is decreasing with $j$ and therefore
\[[n^\theta]\sum_{i=1}^{r-1} P\big(Y_{i[n^\theta]}=0\big) \le \sum_{j=0}^n P(Y_j=0) \lesssim n^{(\alpha-1)/\alpha},\]
the last inequality following from \eqref{eq:mom2:u}.
\end{proof}
It follows from Assumption \ref{cond1} on the characteristic function that the distributions function $F(x)$ of
$X_1$ belongs to the {\it domain of normal attraction} of the symmetric stable law with exponent $\alpha$ (see section 35 of
\cite{gned-kolm}). This means that one has to scale the centered $X_n-\mu n$ by a constant multiple of $n^{1/\alpha}$.
One can characterize such distribution functions.
\begin{lemma}[\cite{gned-kolm}] \label{lem:doa} A necessary and sufficient condition for $F$ to be in the domain of normal attraction of 
a Stable($\alpha$) law, $0<\alpha<2$,  is the existence of contants $ c_1,c_2\ge 0,\, c_1+c_2>0$ such that
\begin{align*}
F(x)&= \big(c_1 + f_1(x)\big) \frac{1}{|x|^\alpha}, \qquad \text{for } x<0, \\
F(x)&= 1-\big(c_2+ f_2(x)\big) \frac{1}{|x|^\alpha}, \;\;\text{for } x>0,
\end{align*}
where the functions $f_1$ and $f_2$ satisfy
\[ \lim_{x \to -\infty} f_1(x) = \lim_{x\to \infty} f_2(x)=0. \]
\end{lemma}
We use the above lemma to deduce the following large deviation estimate.
\begin{lemma} \label{lem:ld} There exists a constant $c_3>0$ such that 
\[ P \big(\vert X_{[nt]-r[n^\theta]}\vert \ge c_3n^\theta\big)\lesssim  \frac{1}{n^{(\alpha-1)\theta}}.\]
\end{lemma}
\begin{proof} For $1<\alpha<2$ we use the result in \cite{heyd}. This gives for $c_3>|\mu|$
\begin{align*}
P(|X_{[nt]-r[n^\theta]}| \ge c_3n^\theta)& \lesssim n^\theta P(|X_1| \ge c_3 n^\theta)\\
&  \lesssim \frac{n^\theta}{n^{\alpha \theta}},
\end{align*}
from the previous lemma. For $\alpha=2$, one can use Exercise 3.3.19 in \cite{durr-fourth}
to conclude that $X_1$ has second moments.  One then uses Chebyschev's inequality to prove the lemma.
\end{proof}

For the rest of this section we provide the proofs of the bounds required for the two terms in \eqref{eq:2*}. Below
is the bound on the first term.
\begin{lemma} \label{lem:2_1} The first term in \eqref{eq:2*} has the bound
\[ \frac{1}{n^{(\alpha-1)/\alpha}} \Big\vert\sum_{i,j} \sum_{k,l} \P_j(l) \cdot \Big[ \P_j(l)-\P_{i[n^\theta]}\big(k[n^\gamma]\big)\Big]\Big\vert \lesssim \frac{1}{n^{(\alpha-1)\theta}} + \frac{n^{\theta+o(1)}}{n^{(\alpha-1)/\alpha}},\]
where the limits of the summation are as in \eqref{limit}.
\end{lemma}
\begin{proof}
We write
\be \label{eq:2:1} \begin{split}
&\frac{1}{n^{(\alpha-1)/\alpha}} \sum_{i,j} \sum_{k,l} \P_j(l) \cdot \Big[ \P_j(l)-\P_{i[n^\theta]}\big(k[n^\gamma]\big)\Big] \\
&= \frac{1}{n^{(\alpha-1)/\alpha}} \sum_{i,j} P(X_j=\tilde X_j)  \\
&\quad - \frac{1}{n^{(\alpha-1)/\alpha}}\sum_{i,j}\sum_{k,l}\sum_{w\in \Z} P\left(X_{i[n^\theta]}=l-w,\, \tilde{X}_{[in^\theta]}=k[n^\gamma]\right) \cdot \P_{j-[in^\theta]}(w), 
\end{split}\ee
the equality holding by an application of the Markov property. We focus on the second term for now, we will return to the first term later. We split the sum according to whether $i=0$ or not. 
The $i=0$ term is of order $n^{\theta}n^{-(\alpha-1)/\alpha}$. For the term corresponding to $i\ge 1$ replace $\tilde X_{i[n^\theta]} = k[n^\gamma]$ by $\tilde X_{i[n^\theta]}=l$ by using Theorem \ref{llt} with $b=1-\alpha\gamma$ to obtain
\bes\begin{split}
& \frac{1}{n^{(\alpha-1)/\alpha}}\sum_{i,j}\sum_{k,l} \sum_w P\left(X_{i[n^\theta]}=l-w,\, \tilde{X}_{[in^\theta]}=k[n^\gamma]\right) \cdot \P_{j-[in^\theta]}(w) \\
&=O\left(\frac{n^\theta}{n^{(\alpha-1)/\alpha}}\right) + \frac{1}{n^{(\alpha-1)/\alpha}}\sum_{i=1}^{r-1}\sum_{j}\left[O\left(\frac{1}{(in^\theta)^{(1+a)/\alpha}} \right)+ O\left(\frac{n^\gamma}{(in^\theta)^{2/\alpha}}\right)\right]\\
&\quad +\frac{1}{n^{(\alpha-1)/\alpha}}\sum_{i=1}^{r-1}\sum_{j}\sum_{k,l} \sum_{w} P\left(X_{i[n^\theta]}=l-w, \tilde{X}_{[in^\theta]}=l\right)\cdot \P_{j-[in^\theta]}(w) \\
&=O\left(\frac{n^\theta}{n^{(\alpha-1)/\alpha}}\right) +O\left(\frac{n^{o(1)}}{n^{\min(a, \alpha-1)/\alpha}}\right)\\
&\quad +\frac{1}{n^{(\alpha-1)/\alpha}}\sum_{i=1}^{r-1}\sum_{j}\sum_{w} P\left(X_{i[n^\theta]}-\tilde{X}_{i[n^\theta]}=w\right)\cdot \P_{j-[in^\theta]}(w).
\end{split}\ees
Returning to our expression \eqref{eq:2:1} we now can write it as
\bes
\begin{split}
&  \frac{1}{n^{(\alpha-1)/\alpha}} \sum_{i,j} P(X_j=\tilde X_j) -\frac{1}{n^{(\alpha-1)/\alpha}}\sum_{i=1}^{r-1} \sum_j\sum_{|w|\le c_3n^\theta} P\left(X_{i[n^\theta]}-\tilde{X}_{i[n^\theta]}=w\right)\cdot \P_{j-[in^\theta]}(w) \\
& +O\left(\frac{1}{n^{(\alpha-1)\theta}}\right)+O\left(\frac{n^\theta}{n^{(\alpha-1)/\alpha}}\right) +O\left(\frac{n^{o(1)}}{n^{\min(a, \alpha-1)/\alpha}}\right), 
\end{split}
\ees
where we have used Lemma \ref{lem:green} and Lemma \ref{lem:ld} with the constant $c_3$ from there. Let us now consider each of the first two terms above. Using Theorem \ref{llt} as well as \eqref{eq:scale} one gets for the first term
\bes
\frac{1}{n^{(\alpha-1)/\alpha}} \sum_{i,j} P(X_j=\tilde X_j) = \frac{\alpha \tilde p_1(0) }{\alpha-1} \cdot \left(\frac{r[n^\theta]-1}{n}\right)^{(\alpha-1)/\alpha} + O\left(\frac{n^{o(1)}}{n^{\min (a, \,\alpha-1)/\alpha}}\right),
\ees
where $\tilde p_1(\cdot)$ is the transition kernel of $-2\nu (-\Delta)^{\alpha/2}$. Using Theorem \ref{llt} again with $b=1-\alpha\theta$, we get for $|w|\le c_3n^\theta$
\bes\begin{split}
&\frac{[n^\theta]}{n^{(\alpha-1)/\alpha}}\sum_{i=1}^{r-1}\sum_{|w|\le c_3[n^\theta]} P\left(X_{i[n^\theta]}-\tilde{X}_{i[n^\theta]}=w\right)\cdot \P_{j-i[n^\theta]}(w) \\
&=\frac{[n^\theta]}{n^{(\alpha-1)/\alpha}}\sum_{i=1}^{r-1}\left[\frac{\tilde p_1(0)}{(i[n^\theta])^{1/\alpha}}+ O\left(\frac{1}{(in^\theta)^{(1+a)/\alpha}}\right)+O\left(\frac{n^\theta}{(in^\theta)^{2/\alpha}}\right)\right] \\
&= \frac{\alpha \tilde p_1(0) }{\alpha-1} \cdot \left(\frac{r[n^\theta]-1}{n}\right)^{(\alpha-1)/\alpha} +O\left(\frac{n^{\theta(\alpha-1)/\alpha}}{n^{(\alpha-1)/\alpha}}\right)+ O\left(\frac{n^{o(1)}}{n^{\min(a,\, \alpha-1)/\alpha}}\right) + O\left(\frac{n^{\theta+o(1)}}{n^{(\alpha-1)/\alpha}}\right).
\end{split}\ees
Collecting all our estimates and recalling our conditions \eqref{g:t:cond} on $\gamma$ and $\theta$ completes the proof.
\end{proof}

We next bound the second term in \eqref{eq:2*}.
\begin{lemma} \label{lem:2_2} The second term in \eqref{eq:2*} has the bound
\[\frac{1}{n^{(\alpha-1)/\alpha}}\Big\vert\sum_{i,j} \sum_{k,l} \P_{i[n^\theta]}\big(k[n^\gamma]\big)\cdot \Big[ \P_{i[n^\theta]}\big(k[n^\gamma]\big)-\P_j(l)\Big]\Big \vert \lesssim \frac{n^{\gamma}}{n^{(\alpha-1)\theta}}+ \frac{n^{\theta+\gamma+o(1)}}{n^{(\alpha-1)/\alpha}} \]
where the limits in the summation are as in \eqref{eq:2*}.
\end{lemma}
\begin{proof} We separate out the $i=0$ term and obtain
\bes\label{eq:2_2} \begin{split}
& \frac{1}{n^{(\alpha-1)/\alpha}} \sum_{i,j} \sum_{k,l} \P_{i[n^\theta]}\big(k[n^\gamma]\big)\cdot \Big[ \P_{i[n^\theta]}\big(k[n^\gamma]\big)-\P_j(l)\Big] \\
&= O \left( \frac{n^{\theta+\gamma}}{n^{(\alpha-1)/\alpha}}\right) + \frac{1}{n^{(\alpha-1)/\alpha}} \sum_{i=1}^{r-1} \sum_{j, k,l} \sum_w \P_{i[n^\theta]}\big(k[n^\gamma]\big)\cdot \Big[ \P_{i[n^\theta]}\big(k[n^\gamma]\big)-\P_{i[n^\theta]}(l-w)\Big]\cdot \P_{j-i[n^\theta]}(w)
\end{split}\ees
The reason for the error bound for the first term is that restricting $i=0$ forces $k=0$, and the number of terms in
the summation over $j$ and $l$ is of order $n^{\theta+\gamma}$.
As in Lemma \ref{lem:2_1} we split the sum over $w$ according to whether $|w|\le c_3[n^\theta]$ or not. In the case $|w|\le c_3[n^\theta]$ we use Theorem \ref{llt} with $b=1-\alpha\theta$. Thus the above is 
\bes\begin{split} \label{eq:2_2:1}
&=O \left( \frac{n^{\theta+\gamma}}{n^{(\alpha-1)/\alpha}}\right)  +\frac{n^{\gamma+\theta}}{n^{(\alpha-1)/\alpha}} \sum_{i=1}^{r-1} \left[O\left(\frac{1}{(in^\theta)^{(1+a)/\alpha}}\right) + O\left(\frac{n^\theta}{(in^\theta)^{2/\alpha}}\right)\right] \\
&\quad + \frac{1}{n^{(\alpha-1)/\alpha}}\sum_{i=1}^{r-1} \sum_{j,k,l}\sum_{|w|\ge c_3n^\theta} \P_{i[n^\theta]}(k[n^\gamma])\cdot\left[\P_{i[n^\theta]}(k[n^\gamma]) -\P_{i[n^\theta]}(l-w)\right]\cdot \P_{j-i[n^\theta]}(w)  \\
&= O\left( \frac{n^{\gamma+o(1)}}{n^{\min(a, \alpha-1)/\alpha}}\right) + O \left(\frac{n^{\gamma+\theta+o(1)}}{n^{(\alpha-1)/\alpha}}\right) \\
&\quad + \frac{1}{n^{(\alpha-1)/\alpha}}\sum_{i=1}^{r-1} \sum_{j,k,l}\sum_{|w|\ge c_3n^\theta} \P_{i[n^\theta]}(k[n^\gamma])\cdot\left[\P_{i[n^\theta]}(k[n^\gamma]) -\P_{i[n^\theta]}(l-w)\right]\cdot \P_{j-i[n^\theta]}(w)  
\end{split}\ees
To bound the last term we ignore the difference in the expression and instead bound the sum. By Lemmas \ref{lem:green}
and \ref{lem:ld}
\bes \label{eq:2_2:2}
\frac{1}{n^{(\alpha-1)/\alpha}}\sum_{i=1}^{r-1} \sum_{j,k,l}\sum_{|w|\ge c_3n^\theta} \P^2_{i[n^\theta]}(k[n^\gamma])\cdot \P_{j-i[n^\theta]}(w)  \lesssim \frac{n^{\gamma}}{n^{(\alpha-1)\theta}},
\ees
and 
\bes\begin{split} \label{eq:2_2:3}
&\frac{1}{n^{(\alpha-1)/\alpha}}\sum_{i=1}^{r-1} \sum_{j,k,l}\sum_{|w|\ge c_3n^\theta} \P_{i[n^\theta]}(k[n^\gamma])\cdot \P_{i[n^\theta]}(l-w)\cdot\P_{j-i[n^\theta]}(w) \\
&\lesssim  \frac{1}{n^{(\alpha-1)/\alpha}}\sum_{i=1}^{r-1}\sum_j\sum_{y=0}^{[n^\gamma]-1}\sum_{|w|\ge c_3n^\theta} P\left( \tilde X_{i[n^\theta]}- X_{i[n^\theta]} = y-w\right)\cdot \P_{j-i[n^\theta]}(w) \\
&\lesssim  \frac{1}{n^{(\alpha-1)/\alpha}}\sum_{i=1}^{r-1}\sum_j\sum_{y=0}^{[n^\gamma]-1}\sum_{|w|\ge c_3n^\theta} P\left( \tilde X_{i[n^\theta]}- X_{i[n^\theta]} = 0\right)\cdot \P_{j-i[n^\theta]}(w)  \\
&\lesssim\frac{n^{\gamma}}{n^{(\alpha-1)\theta}},
\end{split}\ees
 where we used Lemma \ref{lem:green} in the last step. Collecting our estimates completes the proof.
\end{proof}

\noindent\textbf{Acknowledgements:} The author expresses his gratitude to Davar Khoshnevisan for providing 
the proof of Lemma 3.3. He also thanks David Applebaum for comments on an earlier version of the paper. Partial support from EPSRC through grant EP/N028457/1 is gratefully acknowledged.

{\small \bibliography{growthrefs2}
\bibliographystyle{plain}}

	\noindent {\footnotesize MATHEW JOSEPH, SCHOOL OF MATHEMATICS AND STATISTICS, UNIVERSITY OF SHEFFIELD, HOUNSFIELD ROAD, SHEFFIELD - S37RH, U.K., EMAIL: {\tt m.joseph@sheffield.ac.uk}
\end{document}